\documentclass[a4paper,10pt]{article}
\usepackage{amsthm,amsmath,amssymb,graphicx}
\usepackage{hyperref}
\hypersetup{
    colorlinks=true,       
    linkcolor=blue,        
    citecolor=red,         
    filecolor=magenta,     
    urlcolor=cyan,         
    linktocpage=true
}

\renewcommand{\title}[1]{\vspace{\fill}
  \eject\addtolength{\baselineskip}{4pt}
  {\bfseries\LARGE #1}\\[3mm]\addtolength{\baselineskip}{-4pt}}
\renewcommand{\author}[3]{\parbox[t]{75mm}
  {\begin{center}{\scshape #1}\\[3mm] #2\\
      {\ttfamily #3} \end{center}}}

\newtheorem{thm}{\bfseries Theorem}
\newtheorem{lem}[thm]{\bfseries Lemma}        
\newtheorem{remark}[thm]{\bfseries Remark}    
\newtheorem{cor}[thm]{\bfseries Corollary}     
\newtheorem{defn}[thm]{\bfseries Definition}
\newtheorem{cl}[thm]{\bfseries Claim}

\usepackage{algorithm}
\usepackage{algpseudocode}
\usepackage{tikz}
\usepackage{tkz-berge}
\usepackage{mathtools}
\usepackage{subcaption}
\usepackage{bbold}
\usepackage{xfrac}
\usepackage{enumerate}

\pgfdeclarelayer{background}
\pgfdeclarelayer{foreground}
\pgfsetlayers{background,main,foreground}

\tikzset{
    my box/.style = {
        , line cap = round
        , line join = round
    }
}

\newcommand{\highlight}[3]{
    \path [my box, line width = #1, draw = #2] #3;

    \pgfmathsetmacro{\innerlinewidth}{0.9 * #1}
    \path [my box, line width = \innerlinewidth, draw = #2] #3;
}

\usetikzlibrary{snakes}
\tikzset{wavy/.style={decorate,decoration={snake,amplitude=.4mm,segment length=2mm,post length=0mm,pre length=0mm},line width=.5}}

\SetVertexNormal[FillColor=gray!25]

\makeatletter
\renewcommand*{\write@math}[3]{%
            \pgfmathtruncatemacro{\printindex}{#3+1}
            \Vertex[x = #1,y = #2,%
                    L = \cmdGR@cl@prefix\grMathSep{\printindex}]{\cmdGR@cl@prefix#3}}

\def\R{\mathbb R}

\def\Q{\mathbb Q}
\def\Z{\mathbb Z}
\def\N{\mathbb N}
\def\1{\mathbb 1}
\def\hit{\mathcal{H}}

\DeclareMathOperator*{\argmax}{arg\,max}
\DeclareMathOperator*{\argmin}{arg\,min}

\DeclarePairedDelimiter\floor{\lfloor}{\rfloor}

\def\offprints#1{\begingroup
\def\protect{\noexpand\protect\noexpand}\xdef\@thanks{\@thanks
\protect\footnotetext[0]{\unskip\hskip-15pt{\itshape Send offprint requests
to\/}: \ignorespaces#1}}\endgroup\ignorespaces}

\def\@thanks{}

\usepackage[symbol]{footmisc}

\begin{document}

\begin{center}

  \title{Matchings under distance constraints I.\footnote[1]{Supported by the \'UNKP-20-3 New National Excellence Program of the Ministry for Innovation and Technology from the source of the National Research, Development and Innovation Fund. The present paper is an extended and revised version of a conference proceeding presented at ISCO 2020~\cite{DmISCO2020}}} 
  \author{
    P\'eter Madarasi
  }{
    Department of Operations Research \\
    E\"otv\"os Lor\'and University\\
    Budapest, Hungary
  }{
    madarasi@cs.elte.hu
  }


\end{center}


\begin{quote}
  {\bfseries Abstract:}
  This paper introduces the \emph{$d$-distance matching problem}, in which we are given a bipartite graph $G=(S,T;E)$ with $S=\{s_1,\dots,s_n\}$,
  a weight function on the edges and an integer $d\in\Z_+$. The goal is to find a maximum-weight subset $M\subseteq E$ of the edges satisfying the following two conditions: i) the degree of every node of $S$ is at most one in $M$, ii) if $s_it,s_jt\in M$, then $|j-i|\geq d$.
  This question arises naturally, for example, in various scheduling problems.

  $\quad$We show that the problem is NP-complete in general and admits a simple $3$-approxi\-mation. We give an FPT algorithm parameterized by $d$ and also show that the case when the size of $T$ is constant can be solved in polynomial time. From an approximability point of view, we show that the integrality gap of the natural integer programming model is at most $2-\frac{1}{2d-1}$, and give an LP-based approximation algorithm for the weighted case with the same guarantee. A combinatorial $(2-\frac{1}{d})$-approximation algorithm is also presented. Several greedy approaches are considered, and a local search algorithm is described that achieves an approximation ratio of $3/2+\epsilon$ for any constant $\epsilon>0$ in the unweighted case. The novel approaches used in the analysis of the integrality gap and the approximation ratio of locally optimal solutions might be of independent combinatorial interest.
\end{quote}

\begin{quote}
\end{quote}
\vspace{5mm}

\section{Introduction}\label{sec:intro}
In the \emph{perfect $d$-distance matching problem}, one is given a bipartite graph $G=(S,T;E)$ with $S=\{s_1,\dots,s_n\}$, $T=\{t_1,\dots,t_k\}$, a weight function on the edges $w:E\to\R_+$ and an integer $d\in\Z_+$.
The goal is to find a maximum-weight subset $M\subseteq E$ of the edges such that the degree of every node of $S$ is one in $M$ and if $s_it,s_jt\in M$, then $|j-i|\geq d$. In the (non-perfect) \emph{$d$-distance matching problem}, some of the nodes of $S$ might remain uncovered. Note that the order of nodes in $S=\{s_1,\dots,s_n\}$ affects the set of feasible $d$-distance matchings, but the order of $T=\{t_1,\dots,t_k\}$ is indifferent. For example, Figure~\ref{fig:dmFeas} is a feasible \mbox{perfect 3-distance matching}, but the example shown in Figure~\ref{fig:dmInfeas} is not, because edges $s_1t_2$ and $s_3t_2$ violate the $3$-distance condition.

\begin{figure}
  \begin{subfigure}{.49\textwidth}
    \centering
    \begin{tikzpicture}[xscale=.8]
      \SetVertexMath
      \grEmptyPath[form=1,x=0,y=1.5,RA=1.5,rotation=0,prefix=s]{5}
      \grEmptyPath[form=1,x=1.5,y=0,RA=1.5,rotation=0,prefix=t]{3}
      \draw[] (t0)--(s2);
      \draw[] (t1)--(s0);
      \draw[] (t1)--(s3);
      \draw[] (t2)--(s1);
      \draw[] (t2)--(s4);
    \end{tikzpicture} 
    \caption{A feasible perfect $3$-distance matching.}
    \label{fig:dmFeas}
  \end{subfigure}
  \begin{subfigure}{.5\textwidth}
    \centering
    \begin{tikzpicture}[xscale=.8]
      \SetVertexMath
      \grEmptyPath[form=1,x=0,y=1.5,RA=1.5,rotation=0,prefix=s]{5}
      \grEmptyPath[form=1,x=1.5,y=0,RA=1.5,rotation=0,prefix=t]{3}
      \draw[] (t0)--(s3);
      \draw[] (t1)--(s0);
      \draw[] (t1)--(s2);
      \draw[] (t2)--(s1);
      \draw[] (t2)--(s4);
    \end{tikzpicture} 
    \caption{An infeasible $3$-distance matching.}
    \label{fig:dmInfeas}
  \end{subfigure}
  \caption{}
  \label{fig:dmEg}
\end{figure}

An application of this problem for $w\equiv 1$ is as follows. Imagine $n$ consecutive all-day events $s_1,\dots,s_n$ each of which must be assigned one of $k$ watchmen $t_1,\dots,t_k$. For each event $s_i$, a set of possible watchmen is given --- those who are qualified to be on guard at event $s_i$. Appoint exactly one watchman to each of the events such that no watchman is assigned to more than one of any $d$ consecutive events, where $d\in\Z_+$ is given.
In the weighted version of the problem, let $w_{s_it_j}$ denote the level of safety of event $s_i$ if watchman $t_j$ is on watch, and the objective is to maximize the level of overall safety.

As another application of the above question, consider $n$ items $s_1,\dots,s_n$ one after another on a conveyor belt, and $k$ machines $t_1,\dots,t_k$. Each item $s_i$ is to be processed on the conveyor belt by one of the qualified machines ${N(s_i)\subseteq\{t_1,\dots,t_n\}}$ such that if a machine processes item $s_i$, then it can not process the next ${d-1}$~items --- because the conveyor belt is running.


Motivated by the first application, in the \emph{cyclic $d$-distance matching problem} the nodes of $S$ are considered to be in cyclic order. The focus of this paper is on the above (perfect) $d$-distance matching problem, but some of the proposed approaches also apply for the cyclic case. In particular, the $3$-approximation greedy algorithm achieves the same guarantee for the weighted cyclic case (see Section~\ref{sec:greedy}), and the $(3/2+\epsilon)$-approximation algorithm for the unweighted case (see Section~\ref{sec:localSearch}).

\paragraph{\normalfont\textbf{Previous work}}
Observe that in the special case $d=|S|$, one gets the classic (perfect) bipartite matching problem. For $d=1$, the problem reduces to the \mbox{$b$-matching problem}, and one can show that it is a special case of the circulation problem for $d=2$. This implies that the problem is solvable in strongly polynomial time for $d=1,2$, since the circulation problem can be solved in strongly polynomial time~\cite{Tardos1985} and the $b$-matching problem is a special case of it.

A feasible $d$-distance matching $M$ can be thought of as a $b$-matching that contains none of the subgraphs $\{(\{s_i,s_j\},\{t\};\{s_it,s_jt\}) : s_it,s_jt\in E \text{ and } |i-j|\leq d\}$, where $b_s=1$ for $s\in S$ and $b_t=|S|$ for $t\in T$.
A similar problem is the $K_{p,p}$-free $p$-matching problem~\cite{Makai07}. Here one is given an arbitrary family $\mathcal T$ of the subgraphs of $G$ isomorphic with $K_{p,p}$, and the goal is to find a maximum-cardinality $b$-matching which induces no subgraph of $\mathcal T$, where $b:S\cup T\longrightarrow\{0,\dots,p\}$. This problem can be solved in polynomial time. Note that in the distance matching problem, $b$ is different, and the type of the forbidden subgraphs is $K_{2,1}$.

Another similar problem is the following. Given a partition $E_1,\dots,E_k$ of $E$ and positive integers $r_1,\dots,r_k$, find a perfect matching $M$ for which $|M\cap E_i|\leq r_i$. The problem is introduced and shown to be NP-complete in~\cite{Itai78}. Note that the side constraints in the distance matching problem are similar, but the degree constraints are different and our edge sets do not form a partition of $E$.

Several other versions of the "restricted" ($b$-)matching problem have been introduced, for example in~\cite{baste2019uniquelyRestrMatch,BercziVegh10,furst2019acyclicMatching,pap2005alternating}.

The perfect $d$-distance matching problem is a special case of the list-coloring problem on interval graphs~\cite{ZeithoferWess03}. Here a proper vertex coloring must be found for which the color of each node $v$ is chosen from a predefined list of colors $C_v$. Given an instance of the \mbox{$d$-distance matching problem}, we construct an interval \mbox{graph $H=(V,F)$} such that there is a one-to-one correspondence between perfect $d$-distance matchings of $G=(S,T;E)$ and proper list colorings of $H$. Let the nodes of $H$ be the intervals $\{R_d(s):s\in S\}$, where $R_d(s_i)=\{s_i,\dots,s_{\min(i+d-1,|S|)}\}$ for $i=1,\dots,n$, and let two distinct nodes $R_{s_i},R_{s_j}\in V$ be connected by an edge if and only if $R_{s_i}\cap R_{s_j}\neq\emptyset$. Finally, let the list of possible colors of node $R_d(s_i)$ be the neighbors of $s_i$ in $G$.
Observe that two nodes $s_i,s_j\in S$ can be assigned to the same node of $T$ in a distance matching $M$ if and only if there is no edge between $R_d(s_i)$ and $R_d(s_j)$ in $F$. The latter holds, however, if and only if nodes $R_d(s_i)\in V$ and $R_d(s_j)\in V$ can be assigned color $t$ simultaneously in a proper list coloring of $H$ (note that $t$ is in both lists of colors). Hence, there is a one-to-one correspondence between the perfect $d$-distance matchings of $G$ and the proper list colorings of $H$.

The perfect $d$-distance matching problem is also a special case of the frequency assignment problem~\cite{Aardal2007}. Let $S=\{s_1,\dots,s_n\}$ be a set of antennas and let $T=\{t_1,\dots,t_k\}$ be a set of frequencies. There is an edge between $s\in S$ and $t\in T$ if antenna $s$ can be set to frequency $t$. We are also given an interference graph of the antennas, in which two antennas are connected if and only if they may interfere with each other. The goal is to assign a frequency to each antenna such that no two interfering antennas are assigned the same (or in a certain sense similar, see ~\cite{Aardal2007}) frequency. To reduce the $d$-distance matching problem to the frequency assignment problem, let two antennas $s_i,s_j$ interfere if and only if $|i-j|\leq d$. This corresponds to the setting on the plane when antennas $s_1,\dots,s_n$ are located along a straight line in this order such that the Euclidean distance between $s_i$ and $s_{i+1}$ is one for $i=1,\dots,n-1$, and two antennas may interfere if and only if their Euclidean distance is at most $d$. By this construction, there exists a feasible frequency assignment (in which no two interfering antennas are assigned the same frequency) if and only if there exists a perfect $d$-distance matching.

\paragraph{\normalfont\textbf{Our results}} This paper settles the complexity of the distance matching problem and gives an FPT algorithm parameterized by $d$. An efficient algorithm for constant $T$ is also given. We present an LP-based $(2-\frac{1}{2d-1})$-approximation algorithm for the weighted distance matching problem, which implies that the integrality gap of the natural IP model is at most $2-\frac{1}{2d-1}$. An interesting alternative proof for the integrality gap is also given. We also describe a combinatorial $(2-\frac{1}{d})$-approximation algorithm for the weighted case.
One of the main contributions of the paper is a $(3/2+\epsilon)$-approximation algorithm for the unweighted case for any constant $\epsilon>0$ in the unweighted case. The proof is based on revealing the structure of locally optimal solutions recursively. A generalization of K\H onig’s edge-coloring theorem~\cite[Page 74]{AF11} is given to the distance matching problem, as well.
\paragraph{\normalfont\textbf{Notation}} Throughout the paper, assume that $G=(S,T;E)$ contains no loops or parallel edges, unless stated otherwise.
Let $\Delta(v)$ and $N(v)$ denote the set of incident edges to node $v$ and the neighbors of $v$, respectively. For a subset $X\subseteq E$ of the edges, $N_X(v)$ denotes the neighbors of $v$ for edge set $X$.
We use $\deg(v)$ to denote the degree of node $v$.
Let $L_d(s_i)=\{s_{\max(i-d+1,1)},\dots,s_i\}$ and $R_d(s_i)=\{s_i,\dots,s_{\min(i+d-1,|S|)}\}$. The maximum of the empty set is $-\infty$ by definition. Given a function $f:A\to B$, both $f(a)$ and $f_a$ denote the value $f$ assigns to $a\in A$, and let $f(X)=\sum_{a\in X}f(a)$ for $X\subseteq A$. Let $\chi_Z$ denote the characteristic vector of set $Z$, i.e. $\chi_Z(y)=1$ if $y\in Z$, and $0$ otherwise. Occasionally, the braces around sets consisting of a single element are omitted, e.g. $\chi_e=\chi_{\{e\}}$ for $e\in E$.\looseness=-1

\section{Complexity}\label{sec:complexity}
This section settles the complexity of the $d$-distance matching problem. First, we introduce the following NP-complete problem. 
\begin{lem}\label{lem:intermedProblem}
  Given a bipartite graph $G=(S,T;E)$ and $S_1,S_2\subseteq S$ such that ${S_1\cup S_2=S}$, it is NP-complete to decide if there exists $M\subseteq E$ for which $|M|=|S|$ and both $M\cap E_1$ and $M\cap E_2$ are matchings, where $E_i$ denotes the edges induced by $T$ and $S_i$ for $i=1,2$. The problem remains NP-complete even if the maximum degree of the graph is at most 4.
\end{lem}
  
  \begin{proof}
    We reduce the 3-Dimensional Matching problem to the problem defined in the lemma statement.
    Here, one is given three finite disjoint sets $X,Y,Z$ and a set of hyperedges $\mathcal{H}\subseteq X\times Y\times Z$. A subset of the hyperedges $F\subseteq\mathcal{H}$ is called 3-dimensional matching if $x_1\neq x_2, y_1\neq y_2$ and $z_1\neq z_2$ for any two distinct triples $(x_1, y_1, z_1), (x_2, y_2, z_2) \in F$. Being one of Karp's 21 NP-complete problems~\cite{Karp72}, it is NP-complete to decide whether there exists a 3-dimensional matching $F\subseteq\mathcal{H}$ of size $|Z|$. In fact, the problem remains NP-complete even if no element of $X\cup Y\cup Z$ occurs in more than three triples in $\mathcal{H}$~\cite[Page 221]{GareyJohnson79}. Without loss of generality, one might assume that $|X|=|Y|=|Z|$. Let $\mathcal{H}_z=\{e^z_1,\dots,e^z_{k_z}\}$ denote the set of hyperedges incident to $z\in Z$, i.e. $\mathcal{H}_z=\mathcal{H}\cap (X\times Y\times\{z\})$ for each $z\in Z$. To reduce the 3-dimensional matching problem to the above problem, consider the following construction.

    First define a bipartite graph $G=(S,T;E)$ where $S=X\cup (\mathcal{H}\setminus\{e^z_1:z\in Z\})\cup Y$, $T=\mathcal{H}$ and $E$ is as follows. For each $s\in S\cap(X\cup Y)$, add an edge between $s$ and all the hyperedges $e\in T$ incident to $s$; and connect each $e^z_i\in S\cap\mathcal{H}$ to hyperedges $e^z_{i-1},e^z_{i}\in T$ for each $z\in Z$ and $i=2,\dots,k_z$. Let $S_1=S\setminus Y$ and $S_2=S\setminus X$. Figure~\ref{fig:3dimMatching}~and~\ref{fig:3dimMatchingConstruction} show an instance of the 3-dimensional matching problem and the corresponding construction, respectively. Each hyperedge is represented by a unique line style, e.g. the dotted lines represent hyperedge $e_1=(x_2,y_1,z_1)$ in Figure~\ref{fig:3dimMatching}, and the dotted lines correspond to the same hyperedge $e_1$ in Figure~\ref{fig:3dimMatchingConstruction}. Note that the edges represented by a straight line in Figure~\ref{fig:3dimMatchingConstruction} do not represent hyperedges, but the edges between hyperedges. The highlighted edges in Figure~\ref{fig:3dimMatching}~and~\ref{fig:3dimMatchingConstruction} correspond to the same feasible 3-dimensional matching.

    Observe that there exists a 3-dimensional matching $F$ of size $|Z|$ if and only if there exists $M\subseteq E$ for which $|M|=|S|$ and both $M\cap E_1$ and $M\cap E_2$ are matchings, where $E_i$ denotes the edges incident to $S_i$ ($i=1,2$). Indeed, such an $M\subseteq E$ matches all nodes of $S$ into $T$, therefore there exists a unique hyperedge $e^*_z\in T\cap \mathcal H_z$ for each $z\in Z$ that is not matched to $S\cap \mathcal H$, but to exactly one element of $x\in X$ and exactly one element of $y\in Y$ (because all hyperedges in $S\cap \mathcal H_z$ are matched into $T\cap \mathcal H_z$, and these edges of $M$ cover all but one hyperedge of $T\cap \mathcal H_z$). These three edges correspond to the inclusion of hyperedge $(x,y,z)$. This way one obtains a 3-dimensional matching $F$ of size $|Z|$. On the other hand, if a 3-dimensional matching $F$ is given for which $|F|=|Z|$, then one might easily construct the desired $M\subseteq E$ as follows. For each $(x,y,z)\in F$, let $i$ be the unique index such that $e^z_i=(x,y,z)\in F$ and extend $M$ with edges $xe^z_i, ye^z_i\in E$. Let us also include a perfect matching between $S\cap\mathcal H_z$ and $(T\cap \mathcal H_z)\setminus\{e^z_i\}$ (such a perfect matching exists, because the induced subgraph consists of at most two disjoint paths of odd length). It is easy to see that $|M|=|S|$ and both $M\cap E_1$ and $M\cap E_2$ are matchings, and hence $M$ is feasible.

    To complete the proof, observe that the maximum degree in $G$ is at most four if one starts with an instance of the 3-dimensional matching for which no element of $X\cup Y\cup Z$ occurs in more than three triples. Hence, the problem indeed remains NP-complete even if the maximum degree is $4$.
    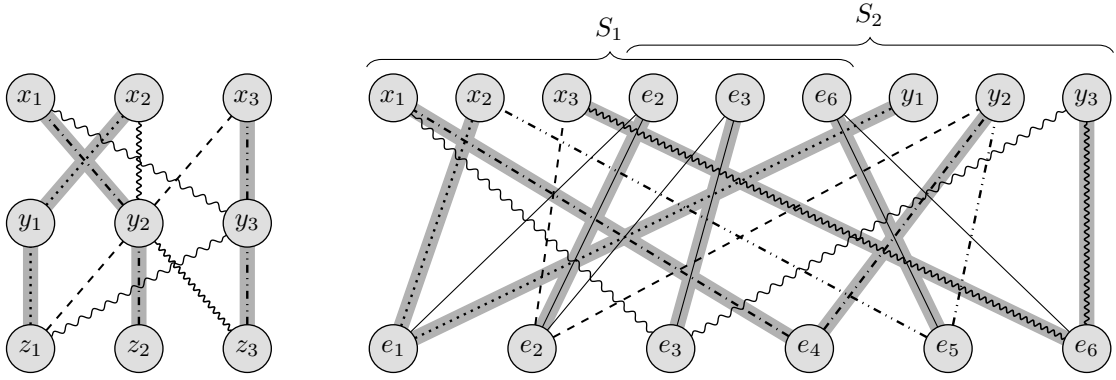
\begin{figure}
      \begin{subfigure}[t]{.29\textwidth}
        \centering
        \begin{tikzpicture}[scale=.95]
          \SetVertexMath
          \grEmptyPath[form=1,x=0,y=0,RA=1.5,rotation=0,prefix=z]{3}
          \grEmptyPath[form=1,x=0,y=1.75,RA=1.5,rotation=0,prefix=y]{3}
          \grEmptyPath[form=1,x=0,y=3.5,RA=1.5,rotation=0,prefix=x]{3}

          \draw[dotted,line width=.9] (z0) -- (y0);
          \draw[dotted,line width=.9] (y0) -- (x1);

          \draw[dashed,line width=.75] (z0) -- (y1);
          \draw[dashed,line width=.75] (y1) -- (x2);

          \draw[wavy] (z0) -- (y2);
          \draw[wavy] (y2) -- (x0);

          \draw[dashdotted,line width=.9] (z1) -- (y1);
          \draw[dashdotted,line width=.9] (y1) -- (x0);
          
          \draw[dashdotdotted,line width=.9] (z2) -- (y2);
          \draw[dashdotdotted,line width=.9] (y2) -- (x2);

          \draw[decorate,decoration={snake,amplitude=.3mm,segment length=1mm,post length=0mm,pre length=0mm},line width=.6] (z2) -- (y1);
          \draw[decorate,decoration={snake,amplitude=.3mm,segment length=1mm,post length=0mm,pre length=0mm},line width=.6] (y1) -- (x1);


          \begin{pgfonlayer}{background}
            \highlight{2mm}{black!30}{(z0.center) -- (y0.center) -- (x1.center)}
            \highlight{2mm}{black!30}{(z1.center) -- (y1.center) -- (x0.center)}
            \highlight{2mm}{black!30}{(z2.center) -- (y2.center) -- (x2.center)}
          \end{pgfonlayer}
          
        \end{tikzpicture} 
        \caption{An instance of the 3-dimensional matching problem.}
        \label{fig:3dimMatching}
      \end{subfigure}
      \hfill
      \begin{subfigure}[t]{.69\textwidth}
        \centering
        \begin{tikzpicture}[scale=.95]
          \SetVertexMath
          \grEmptyPath[form=1,x=0,y=3.5,RA=1.2,rotation=0,prefix=x]{3}
          \Vertex[x=3.6,y=3.5,L=e_2]{se2}
          \Vertex[x=4.8,y=3.5,L=e_3]{se3}
          \Vertex[x=6,y=3.5,L=e_6]{se6}
          \grEmptyPath[form=1,x=7.2,y=3.5,RA=1.2,rotation=0,prefix=y]{3}
          
          \grEmptyPath[form=1,x=0,y=0,RA=8*1.2/5,rotation=0,prefix=e]{6}

          \draw[dotted,line width=.9] (e0) -- (x1);
          \draw[dotted,line width=.9] (e0) -- (y0);

          \draw[dashed,line width=.75] (e1) -- (y1);
          \draw[dashed,line width=.75] (e1) -- (x2);

          \draw[wavy] (e2) -- (y2);
          \draw[wavy] (e2) -- (x0);

          \draw[dashdotted,line width=.9] (e3) -- (y1);
          \draw[dashdotted,line width=.9] (e3) -- (x0);

          \draw[dashdotdotted,line width=.9] (e4) -- (y1);
          \draw[dashdotdotted,line width=.9] (e4) -- (x1);

          \draw[decorate,decoration={snake,amplitude=.3mm,segment length=1mm,post length=0mm,pre length=0mm},line width=.6] (e5) -- (y2);
          \draw[decorate,decoration={snake,amplitude=.3mm,segment length=1mm,post length=0mm,pre length=0mm},line width=.6] (e5) -- (x2);
          \draw[] (e0) -- (se2);
          \draw[] (e1) -- (se2);
          \draw[] (e1) -- (se3);
          \draw[] (e2) -- (se3);
          \draw[] (e4) -- (se6);
          \draw[] (e5) -- (se6);

          \begin{pgfonlayer}{background}
            \highlight{2mm}{black!30}{(e0.center) -- (x1.center)}
            \highlight{2mm}{black!30}{(e0.center) -- (y0.center)}
            \highlight{2mm}{black!30}{(se2.center) -- (e1.center)}
            \highlight{2mm}{black!30}{(se3.center) -- (e2.center)}

            \highlight{2mm}{black!30}{(e3.center) -- (x0.center)}
            \highlight{2mm}{black!30}{(e3.center) -- (y1.center)}

            \highlight{2mm}{black!30}{(e5.center) -- (x2.center)}
            \highlight{2mm}{black!30}{(e5.center) -- (y2.center)}
            \highlight{2mm}{black!30}{(e4.center) -- (se6.center)}
          \end{pgfonlayer}
         
        \node (x0-fit) at (0-.5,3.8) {};
        \node (se6-fit) at (6+.5,3.8) {};
        \node (se2-fit) at (3.6-.5,4) {};
        \node (y2-fit) at (7.2+1.2+1.2+.5,4) {};
        \begin{scope}[decoration={brace,amplitude=2mm}]
        \begin{scope}[every node/.style={midway,left,yshift=5mm,xshift=3mm}]
            \draw[decorate] (x0-fit.north east) -- (se6-fit.north west)node(C1){$S_1$};
            \draw[decorate] (se2-fit.north east) -- (y2-fit.north west)node(C2){$S_2$};
        \end{scope}
        \end{scope}
          
        \end{tikzpicture} 
        \caption{The corresponding instance of the problem stated in Claim~\ref{lem:intermedProblem}, where $S_1=\{x_1,x_2,x_3,e_2,e_3,e_6\}$ and $S_2=\{e_2,e_3,e_6,y_1,y_2,y_3\}$.}
        \label{fig:3dimMatchingConstruction}
      \end{subfigure}
      \caption{Illustration of the proof of Claim~\ref{lem:intermedProblem}. Each hyperedge is represented by a unique line style. The highlighted hyperedges on (a) and the highlighted edges on (b) correspond to the same feasible solution.}
      \label{fig:dmEg2}
    \end{figure}
    \qed
  \end{proof}

In what follows, the previous problem is reduced to the $d$-distance matching problem, hence the hardness of the latter.

\begin{thm}\label{thm:dmnpc}
  It is NP-complete to decide if a graph has a perfect $d$-distance matching, even if the maximum degree of the graph is at most 4.
\end{thm}
  \begin{proof}
    It suffices to reduce the problem from Lemma~\ref{lem:intermedProblem} to the perfect $d$-distance matching problem. Let $G=(S,T;E)$; $S_1,S_2\subseteq S$,  $S_1\cup S_2=S$ be an instance of the above problem. Without loss of generality, one might assume that $S_1\not\subseteq S_2$ and $S_2\not\subseteq S_1$.
    To construct an instance $G'=(S',T';E'),d\in\N$ of the perfect $d$-distance matching problem, let $G'=G$ and modify $G'$ as follows. Order the nodes of $S'$ such that the nodes of $S_1\setminus S_2$, $S_1\cap S_2$ and $S_2\setminus S_1$ appear in this order (the order of the elements inside the three sets is arbitrary). Insert $|S_1\setminus S_2|$ and $|S_2\setminus S_1|$ new nodes to $S'$ right after the last node of $S_1$ and right before the first node covered by $S_2$, respectively. Finally, add $|S_1\setminus S_2|+|S_2\setminus S_1|$ new nodes to $T'$, extend $E'$ with the edges of a perfect matching between the newly added nodes and let $d=|S|$. Figure~\ref{fig:dmNPCConstr} illustrates the construction. The blank nodes on the figure are the newly inserted ones, and $S_i'$ is the union of $S_i$ and the $i^{\text{th}}$ set of new nodes added to $S'$. The highlighted edges correspond to those in Figure~\ref{fig:3dimMatchingConstruction}.

    \begin{figure}[t]
      \centering
      \begin{tikzpicture}[xscale=.9]
        \SetVertexMath
        \grEmptyPath[form=1,x=0,y=3.5,RA=1.2,rotation=0,prefix=x]{3}
        \Vertex[x=3.6,y=3.5,L=$$]{dummy1}
        \Vertex[x=4.4,y=3.5,L=$$]{dummy2}
        \Vertex[x=5.2,y=3.5,L=$$]{dummy3}
        \Vertex[x=5.2+1.2,y=3.5,L=e_2]{se2}
        \Vertex[x=5.2+2*1.2,y=3.5,L=e_3]{se3}
        \Vertex[x=5.2+3*1.2,y=3.5,L=e_6]{se6}
        \Vertex[x=5.2+4*1.2,y=3.5,L=$$]{dummy4}
        \Vertex[x=5.2+4*1.2+.8,y=3.5,L=$$]{dummy5}
        \Vertex[x=5.2+4*1.2+2*.8,y=3.5,L=$$]{dummy6}
        
        \grEmptyPath[form=1,x=5.2+4*1.2+2*.8+1.2,y=3.5,RA=1.2,rotation=0,prefix=y]{3}

        \pgfmathsetmacro{\xOffsetT}{((5.2+4*1.2+2*.8+1.2+2*1.2)-(1.6+8*1.1+2*1.2+1.6))/2}
        \Vertex[x=\xOffsetT+0,y=0,L=$$]{dummy1T}
        \Vertex[x=\xOffsetT+.8,y=0,L=$$]{dummy2T}
        \Vertex[x=\xOffsetT+1.6,y=0,L=$$]{dummy3T}
        \grEmptyPath[form=1,x=\xOffsetT+1.6+1.2,y=0,RA=8*1.1/5,rotation=0,prefix=e]{6}
        \Vertex[x=\xOffsetT+1.6+8*1.1+2*1.2,y=0,L=$$]{dummy4T}
        \Vertex[x=\xOffsetT+1.6+8*1.1+2*1.2+.8,y=0,L=$$]{dummy5T}
        \Vertex[x=\xOffsetT+1.6+8*1.1+2*1.2+1.6,y=0,L=$$]{dummy6T}

        \Edges(dummy1,dummy1T);
        \Edges(dummy2,dummy2T);
        \Edges(dummy3,dummy3T);
        \Edges(dummy4,dummy4T);
        \Edges(dummy5,dummy5T);
        \Edges(dummy6,dummy6T);

        \draw[dotted,line width=.9] (e0) -- (x1);
        \draw[dotted,line width=.9] (e0) -- (y0);

        \draw[dashed,line width=.75] (e1) -- (y1);
        \draw[dashed,line width=.75] (e1) -- (x2);

        \draw[wavy] (e2) -- (y2);
        \draw[wavy] (e2) -- (x0);

        \draw[dashdotted,line width=.9] (e3) -- (y1);
        \draw[dashdotted,line width=.9] (e3) -- (x0);

        \draw[dashdotdotted,line width=.9] (e4) -- (y1);
        \draw[dashdotdotted,line width=.9] (e4) -- (x1);

        \draw[decorate,decoration={snake,amplitude=.3mm,segment length=1mm,post length=0mm,pre length=0mm},line width=.6] (e5) -- (y2);
        \draw[decorate,decoration={snake,amplitude=.3mm,segment length=1mm,post length=0mm,pre length=0mm},line width=.6] (e5) -- (x2);
        \draw[] (e0) -- (se2);
        \draw[] (e1) -- (se2);
        \draw[] (e1) -- (se3);
        \draw[] (e2) -- (se3);
        \draw[] (e4) -- (se6);
        \draw[] (e5) -- (se6);

        \begin{pgfonlayer}{background}
          \highlight{2mm}{black!30}{(e0.center) -- (x1.center)}
          \highlight{2mm}{black!30}{(e0.center) -- (y0.center)}
          \highlight{2mm}{black!30}{(se2.center) -- (e1.center)}
          \highlight{2mm}{black!30}{(se3.center) -- (e2.center)}

          \highlight{2mm}{black!30}{(e3.center) -- (x0.center)}
          \highlight{2mm}{black!30}{(e3.center) -- (y1.center)}

          \highlight{2mm}{black!30}{(e5.center) -- (x2.center)}
          \highlight{2mm}{black!30}{(e5.center) -- (y2.center)}
          \highlight{2mm}{black!30}{(e4.center) -- (se6.center)}
        \end{pgfonlayer}

        \node (x0-fit) at (0-.5,3.8) {};
        \node (se6-fit) at (5.2+3*1.2+.5,3.8) {};
        \node (se2-fit) at (5.2+1.2-.5,4) {};
        \node (y2-fit) at (5.2+4*1.2+2*.8+1.2+1.2+1.2+.5,4) {};
        \begin{scope}[decoration={brace,amplitude=2mm}]
        \begin{scope}[every node/.style={midway,left,yshift=5mm,xshift=3mm}]
            \draw[decorate] (x0-fit.north east) -- (se6-fit.north west)node(C1){$S_1'$};
            \draw[decorate] (se2-fit.north east) -- (y2-fit.north west)node(C2){$S_2'$};
        \end{scope}
        \end{scope}
        
      \end{tikzpicture} 
      \caption{Illustration of the construction in the proof of Theorem~\ref{thm:dmnpc} for the problem instance presented in Figure~\ref{fig:3dimMatchingConstruction}. There exists a perfect $9$-distance matching if and only if the problem given in Figure~\ref{fig:3dimMatchingConstruction} has a feasible solution of size 9.}
      \label{fig:dmNPCConstr}
    \end{figure}
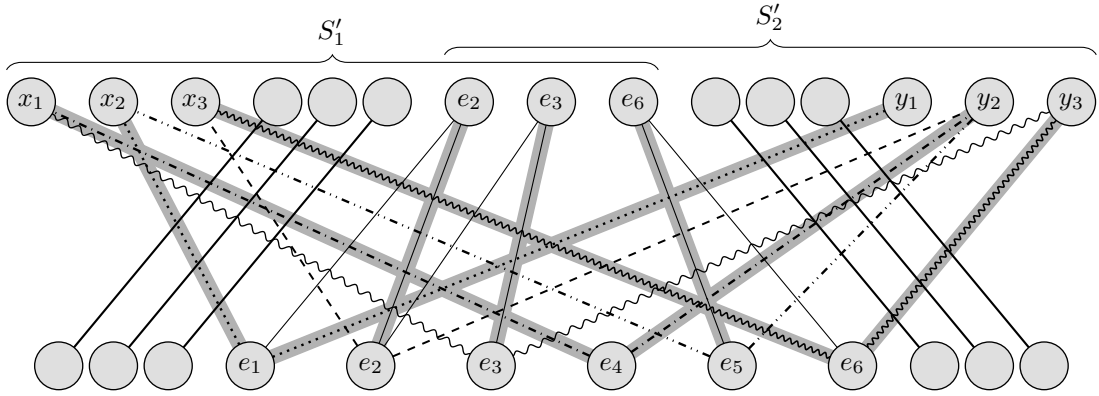

    To complete the proof, observe that there exists a perfect $|S|$-distance matching in $G'$ if and only if there exists $M\subseteq E$ for which $|M|=|S|$ and both $M\cap E_1$ and $M\cap E_2$ are matchings. Indeed, from $M$ one obtains a perfect $|S|$-distance matching in $G'$ by simply adding the perfect matching between the new nodes. To see the other direction, one has to remove this perfect matching from the perfect $S$-distance matching. Note that the maximum degree in $G'$ is not larger than in $G$, hence the problem remains hard even if the maximum degree is at most 4.
    \qed
  \end{proof}

\section{Weighted $d$-distance matching problem}\label{sec:wdm}
This section presents various approaches to the weighted $d$-distance matching problem. Section~\ref{sec:fptD} presents an FPT algorithm~\cite{ParameterizedComplexity} parameterized by $d$, while Section~\ref{sec:constantT} settles the case when the size of $T$ is constant. A simple greedy approach is presented in Section~\ref{sec:greedy}. Finally, Sections~\ref{sec:gap}~and~\ref{sec:lpApxAlgo} are devoted to the investigation of the natural linear programming model.
\subsection{FPT algorithm parameterized by $d$}\label{sec:fptD}
In what follows, an FPT algorithm parameterized by $d$ is presented for the weighted (perfect) $d$-distance matching problem. Observe that the weighted $d$-distance matching problem easily reduces to the perfect case by adding a new node $t_s$ to $T$ and a new edge $st_s$ of weight zero for each $s\in S$, therefore the algorithm is given only for the weighted perfect $d$-distance matching problem. The next claim gives a way to reduce the problem so that it admits an  efficient dynamic programming solution.

\begin{cl}\label{cl:reduction}
    Suppose that $s\in S$ is such that $deg(s)\geq 2d$. Then, we can remove an arbitrary minimum-weight edge of $\Delta(s)$ from the edge set without changing the weight of the optimal perfect $d$-distance matching.
\end{cl}
  \begin{proof}
    Let $st$ be a minimum-weight edge incident to node $s$. In order to prove that $st$ can be removed, it suffices to show that there is a maximum-weight $d$-distance matching that does not use edge $st$. Given a $d$-distance matching $M$ that contains edge $st$, let $Z\subseteq T$ denote the nodes that $M$ assigns to $L_d(s)\cup R_d(s)$.
    Since $|Z|\leq 2d-1$, there exists a node $t'\in N(s)\setminus Z$ for which $w_{st}\leq w_{st'}$. Observe that $M'=(M\cup\{st'\})\setminus\{st\}$ is a perfect $d$-distance matching of weight at least $w(M)$, which does not contain edge $st$. Indeed, the degree of $s$ remains one, and the only edge $M'$ contains between nodes $L_d(s)\cup R_d(s)$ and $t'$ is $st'$ itself (by contradiction, if there were another edge $s't'\in M'\setminus\{st'\}$ for some $s'\in L_d(s)\cup R_d(s)$, then $s't'$ would be in $M$, and hence $t'\in Z$ would hold).
  \qed
  \end{proof}

Based on this claim, the problem can be reduced so that the degree of each node $s\in S$ is at most $2d-1$. The reduction can be performed in $\mathcal{O}(m+n)$ steps by removing all but the $2d-1$ heaviest edges incident to each node $s\in S$. To this end, let us assume that each edge weight occurs only once (otherwise fix an arbitrary order between the ties), and for each $s\in S$, find the $(2d-1)^\text{th}$ lightest edge $e_s$ in $\Delta(s)$ with the linear time selection algorithm, and then eliminate all edges of $\Delta(s)$ which are lighter than $e_s$.

In what follows, a dynamic programming approach is presented to solve the reduced problem in $\mathcal{O}((2d-1)^{d+1}n)$ steps.

For $i\geq d$, let $f(s_i,z_1,\dots,z_d)$ denote the weight of the maximum-weight perfect $d$-distance matching if the problem is restricted to the first $i$ nodes of $S$  and $s_{i-j+1}$ is assigned to its neighbor $z_j$ for $j=1,\dots,d$. Formally, let $$f(s_i,z_1,\dots,z_d)=-\infty$$ if $z_1,\dots, z_d$ are not distinct, otherwise, $f(s_i,z_1,\dots,z_d)$ can be defined by the following recursive formula.
\begin{equation}\label{defOfF}
  f(s_i,z_1,\dots,z_d)=
  \begin{cases}
    \sum\limits_{j=1}^d w_{s_{d-j+1}z_j}        & \text{if $i=d$}\\
    w_{s_iz_1}+\max\limits_{t\in N(s_{i-d})}f(s_{i-1},z_2,\dots,z_d,t)    & \text{if $i>d$},
  \end{cases}
\end{equation}
where $i\geq d$, $s_i\in S$, $z_j\in N(s_{i-j+1})$ for $j=1,\dots,d$ and $z_1,\dots, z_d$ are distinct. To see that recursion (\ref{defOfF}) holds, observe that in its first case, by definition, $f(s_d,z_1,\dots,z_d)$ is the weight of matching $\{s_jz_{d-j+1} : j=1,\dots,d\}$. In the second case of (\ref{defOfF}), $s_i$ must be mapped to $z_1$, and we want to find the maximum-weight perfect $d$-distance matching on the first $i-1$ nodes of $S$ which maps $s_{i-j+1}$ to its neighbor $z_j$ for $j=2,\dots,d$. To this end, we want to find a node $t\in N(s_{i-d})$ (to be assigned to node $s_{i-d}$) which maximizes $f(s_{i-1},z_2,\dots,z_d,t)$.

By definition, the weight of the optimal $d$-distance matching is
\begin{equation}\label{defOfOptF}
  \max\{f(s_n,z_1,\dots,z_d) : z_j\in N(s_{n-j+1}) \text{ for } j=1,\dots,d\}.
\end{equation}

Observe that the number of subproblems is $\mathcal{O}(n(2d-1)^d)$, since the degree of each $s\in S$ is at most $2d-1$. Recursion~(\ref{defOfF}) gives a way to compute $f(s_i,z_1,\dots,z_d)$ in $\mathcal{O}(d)$ steps if the subproblems are computed in appropriate order, i.e. the value $f(s_{i-1},z_1',\dots,z_d')$ is available for all necessary $z_1',\dots,z_d'\in T$. Therefore the number of steps to compute all the subproblems is $\mathcal{O}(dn(2d-1)^d)$. Furthermore, the optimum value can be computed in $\mathcal{O}((2d-1)^d)$ steps by (\ref{defOfOptF}). The overall running time of the algorithm is $\mathcal{O}(dn(2d-1)^d+poly(|S|+|T|))$.

\begin{remark}
  To improve the running time to $O(nd^{d+1}+poly(|S|+|T|))$, observe that when $$\max\limits_{t\in N(s_{i-d})}f(s_{i-1},z_2,\dots,z_d,t)$$ is computed in (\ref{defOfF}), we need to consider only the (at most) $d$ heaviest edges of $\Delta(s_{i-d})$ which are not incident to any of $z_2,\dots,z_d\in T$, since we only need to make sure that there is no conflict with the $d-1$ nodes on the left of $s_{i-d}$. This way the number of subproblems is $O(nd^d)$, and the overall number of steps is $O(nd^{d+1}+poly(|S|+|T|))$. Similarly, in~(\ref{defOfOptF}) one needs to consider only the $d$ heaviest edges incident to $z_j$ which are not incident to any of $z_{j+1},\dots,z_{d}$, therefore there are at most $d^d$ different configurations to be taken into account in~(\ref{defOfOptF}).
\end{remark}

\subsection{Efficient algorithm for  constant $|T|$}\label{sec:constantT}
If the size of $T$ is constant, then one can solve the problem efficiently as well. First, consider the following subproblems. Let $f(s_i,d_1,\dots,d_{|T|})$ denote the weight of the optimal perfect $d$-distance matching when the problem is restricted to $s_1,\dots,s_i$, and $t_j$ cannot be matched to nodes $s_{i-d_j+1},\dots,s_i$ for $j=1,\dots,|T|$ (here $d_j=0$ means that $t_j$ can be matched to any node). Formally, $f(s_i,d_1,\dots,d_{|T|})$ can be defined as follows. If $i\geq 2$, then let
\begin{multline}\label{defOfF2a}
  f(s_i,d_1,\dots,d_{|T|})\\=\max\limits_{t_j\in N(s_i) : d_j=0}\{w_{s_it_j}+f(s_{i-1},d_1',\dots,d_{j-1}',d-1,d_{j+1}',\dots,d_{|T|}')\},
\end{multline}
where $d_k'=\max(d_k-1,0)$ for $k=1,\dots,|T|$.
If $i=1$, then let
\begin{equation}\label{defOfF2b}
  f(s_1,d_1,\dots,d_{|T|})=
  \max\limits_{t_j\in N(s_1) : d_j=0}w_{s_1t_j}.
\end{equation}

By definition, the weight of the optimal $d$-distance matching is given by
\begin{equation}\label{defOfOptF2}
  \max_{t_i\in N(s_n)} f(s_{n-1},0,\dots,0,\underbrace{d-1}_{i^{\text{th}}},0,\dots,0).
\end{equation}

The number of subproblems to be solved is $\mathcal{O}(nd^{|T|})$, each of which can be computed in $\mathcal{O}(|T|)$ steps by (\ref{defOfF2a})~and~(\ref{defOfF2b}). Once all the subproblems are computed, it takes additional $\mathcal{O}(|T|)$ steps to compute the optimal value by (\ref{defOfOptF2}). Hence the overall number of steps is $\mathcal{O}(n|T|d^{|T|})$.

A similar approach settles the non-perfect case for constant $|T|$, the details of which are left to the reader.



\subsection{A greedy algorithm}\label{sec:greedy}
This section describes a greedy method for the weighted (not-necessarily-perfect) $d$-distance matching problem, and proves that it is a 3-approximation algorithm.
\begin{algorithm}
  \caption{\hspace{0.5cm}\textsc{Greedy}}
  \label{alg:greedy}
  \begin{algorithmic}
    \State Let $e_1,\dots,e_m$ be the edges in non-increasing order by their weights.
    \State $M:=\emptyset$
    \For{$i=1,\dots,m$}
    \If{$M\cup\{e_i\}$ is a feasible $d$-distance matching}
    \State $M:=M\cup\{e_i\}$
    \EndIf
    \EndFor
    \State \textbf{output} $M$
  \end{algorithmic}
\end{algorithm}

\begin{thm}\label{thm:wGreedy}
  \textsc{Greedy} is a 3-approximation algorithm for the weighted $d$-dis\-tan\-ce matching problem.
\end{thm}
  
  \begin{proof}
    Assume that \textsc{Greedy} returns edges $f_1,\dots,f_p$, and it selects them in this order.
    Let $M_i$ denote a maximum-weight $d$-distance matching that contains $f_1,\dots,f_i$, where $0\leq i\leq p$, i.e. $$M_i=\argmax\{w(M) : f_1,\dots,f_i\in M \text{ and $M$ is a $d$-distance matching}\}.$$ Furthermore, let $\theta_i$ denote the weight of $M_i$ for $i=0,\dots,p$. Note that $\theta_0$ is the weight of the optimal $d$-distance matching and $\theta_p$ is the weight of the matching \textsc{Greedy} returns. Observe that there exist edges $e,e',e''\in M_i\setminus\{f_1,\dots,f_i\}$ such that $(M_i\setminus\{e,e',e''\})\cup\{f_{i+1}\}$ is a feasible $d$-distance matching, which contains edges $f_1,\dots,f_{i+1}$. By the greedy selection rule, $w_e,w_{e'},w_{e''}\leq w_{f_{i+1}}$, one gets
    \begin{equation}\label{eq:greedyProgress}
      \theta_{i+1}\geq\theta_i+w_{f_{i+1}}-w_{e}-w_{e'}-w_{e''}\geq\theta_i-2w_{f_{i+1}}
    \end{equation}
    holds for all $i=0,\dots,p-1$. A simple inductive argument shows that (\ref{eq:greedyProgress}) implies $\theta_p\geq\theta_0-2\sum\limits_{i=1}^pw_{f_i}$, therefore $3\theta_p\geq\theta_0$ follows, which completes the proof.

    The analysis is tight even for $d=2$ and $w\equiv 1$ in the sense that \textsc{Greedy} might return only one edge, while the largest $2$-distance matching consists of 3 edges, see Figure~\ref{fig:wGreedyTight} for an example.
    \qed
  \end{proof}

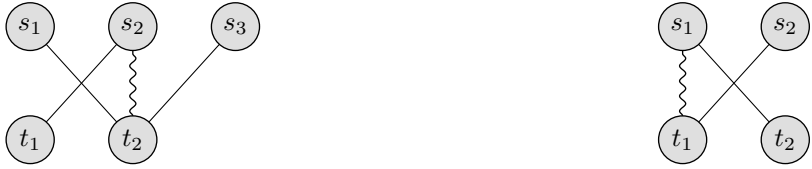
\begin{figure}[H]
  \centering
  \begin{subfigure}[t]{.49\textwidth}
    \centering
    \begin{tikzpicture}[xscale=.9]
      \SetVertexMath
      \grEmptyPath[form=1,x=0,y=1.5,RA=1.5,rotation=0,prefix=s]{3}
      \grEmptyPath[form=1,x=0,y=0,RA=1.5,rotation=0,prefix=t]{2}
      \draw[] (t1)--(s0);
      \draw[] (t1)--(s2);
      \draw[] (t0)--(s1);
      \draw[wavy] (t1) -- (s1);
    \end{tikzpicture} 
    \caption{For $d=2$ and unit weights, \textsc{Greedy} might select edge $s_2t_2$ only, while the largest $2$-distance matching is of cardinality $3$.}
    \label{fig:wGreedyTight}
  \end{subfigure}
  \hfill
  \begin{subfigure}[t]{.49\textwidth}
    \centering
    \begin{tikzpicture}[xscale=.9]
      \SetVertexMath
      \grEmptyPath[form=1,x=0,y=1.5,RA=1.5,rotation=0,prefix=s]{2}
      \grEmptyPath[form=1,x=0,y=0,RA=1.5,rotation=0,prefix=t]{2}
      \draw[] (t1)--(s0);
      \draw[] (t0)--(s1);

      \draw[wavy] (t0) -- (s0);
    \end{tikzpicture} 
    \caption{For $d=2$ and unit weights, both \textsc{$S$-Greedy} and \textsc{$T$-Greedy} select edge $s_1t_1$ only, while the largest $2$-distance matching is of cardinality $2$.}
    \label{fig:unwGreedyTight}
  \end{subfigure}
  \caption{Tight examples for Theorems~\ref{thm:wGreedy},~\ref{thm:SGreedy}~and~\ref{thm:TGreedy}.}
  \label{fig:dmEg3}
\end{figure}

\begin{remark}
  The above proof shows that \textsc{Greedy} is a 3-approximation algorithm for the more general \emph{cyclic $d$-distance matching problem}, in which the nodes of $S$ are considered in cyclic order.
\end{remark}

\subsection{Linear programming}\label{sec:lp}
The following two sections prove that the integrality gap of the natural integer programming model is at most $2-\frac{1}{2d-1}$, and present an LP-based \mbox{$(2-\frac{1}{2d-1})$-ap}proxi\-mation algorithm for the weighted $d$-distance matching problem.
First consider the relaxation of the natural $0-1$ integer programming formulation of the weighted $d$-distance matching problem.

\begin{subequations}
  \begin{align}
    \label{dmLp}
    \tag{LP1}
    \max\sum_{st\in E}&w_{st}x_{st}\\
    \mbox{s.t.}\quad\quad\quad\quad\quad&&\nonumber\\
    x&\in\R_+^{E}&\label{eq:int}\\
    \sum_{st\in\Delta(s)} x_{st} &\leq 1 &\forall s\in S \label{eq:gyerekekDM}\\
    \sum_{s't\in E: s'\in R_d(s)} x_{s't} &\leq 1&\forall s \in S, t \in T \label{eq:autokDM}
  \end{align}
\end{subequations}

One gets the relaxation of the $0-1$ integer programming formulation (LP2)\\of the weighted perfect $d$-distance matching problem by tightening (\ref{eq:gyerekekDM}) to equality in~\ref{dmLp}.




\subsubsection{Integrality gap}\label{sec:gap}
This section proves that the integrality gap of \ref{dmLp} is at most $2-\frac{1}{2d-1}$, and proves the integrality of \ref{dmLp} and 
LP2 in special cases. The former result also follows from the LP-based approximation algorithm described in Section~\ref{sec:lpApxAlgo}. The following definition plays a central role both in the analysis of the integrality gap and in the LP-based approximation algorithm presented in the next section.

\begin{defn}\label{def:flatOrder}
  Given a feasible solution $x$ of \ref{dmLp}, an order of the edges $e_1=s^1t^1,\dots,e_m=s^mt^m$ is \textbf{$\theta$-flat with respect to $x$} if
  \begin{equation}\label{eq:flat}
    \xi_i+\bar\xi_i\leq \theta-x_{e_i}
  \end{equation}
  holds for each $i=1,\dots,m$, where
  $\xi_i=\sum\{x_{e_j} : j>i, e_j\in\Delta(s^i)\}$ and $\bar\xi_i=\sum\{x_{e_j} : j>i, e_j\in\Delta(t^i), s^j\in L_d(s^i)\cup R_d(s^i)\}$.
\end{defn}
That is, an order of the edges is $\theta$-flat if the sum of $x$ on those edges among $e_{i+1},\dots,e_m$ that are hit by an edge $e_i$ is at most $\theta-x_{e_i}$ for every $i$. Note that any order of the edges is $3$-flat by definition, since for any edge $e=st$, the sum of variables on all edges incident to $s$ is at most $1$ by~(\ref{eq:gyerekekDM}), whereas the sum on the edges induced by $L_d(s)\cup R_d(s)$ and~$\{t\}$ is at most $2$ by~(\ref{eq:autokDM}). The following lemma further improves this bound to $2-\frac{1}{2d-1}$.

\begin{lem}\label{lem:orderingNEW}
  There exists an optimal solution $x\in\Q^E$ of \ref{dmLp} and an order $e_1=s^1t^1,\dots,e_m=s^mt^m$ of the edges that is $(2-\frac{1}{2d-1})$-flat with respect to $x$.
\end{lem}
  \begin{proof}
    Let $E_s\subseteq\Delta(s)$ denote the first $\min(2d-1,\deg(s))$ largest weight edges incident to node $s$ for each $s\in S$. Let $x$ be an optimal solution to \ref{dmLp} for which $\gamma(x)=\sum\{ x_e : e\in E\setminus\bigcup_{s\in S} E_s\}$ is minimal. Towards a contradiction, suppose that ${\gamma(x)>0}$. By definition, ${\gamma(x)>0}$ implies that there exists an edge $st\in E\setminus\bigcup_{k=1}^n E_k$ for which $x_{st}>0$. There exists an edge $st'\in E_s$ such that 
      $x'=x-\epsilon\chi_{st}+\epsilon\chi_{st'}$
    is feasible for sufficiently small $\epsilon>0$, otherwise $x(\bigcup\{\Delta(s') : s'\in L_d(s)\cup R_d(s)\})\geq 2d-1+\epsilon$ would hold, which is not possible because of the constraints~(\ref{eq:gyerekekDM}).
    But then $wx\leq wx'$ and $\gamma(x')<\gamma(x)$, contradicting the minimality of $\gamma(x)$. Therefore $\gamma(x)=0$ follows, meaning that $x_e=0$ holds for each $e\in E\setminus\bigcup_{s\in S} E_s$. Hence one can restrict the edge set to $\bigcup_{s\in S} E_s$ without change in the optimal objective value, which implies that there exists a rational optimal solution $x\in\Q^E$ of \ref{dmLp} with $\gamma(x)=0$.
    
    Let $x$ be as above, and let $e_1=s^1t^1,\dots,e_m=s^mt^m$ be the order of the edges given by Algorithm~\ref{alg:orderingProc} for input $x$.
    \begin{algorithm}[b]
  \caption{\hspace{0.5cm}The ordering procedure for Lemma~\ref{lem:orderingNEW}}
  \label{alg:orderingProc}
  \begin{algorithmic}

    \State Let $x$ be a given fractional solution to \ref{dmLp} and let $G=(S,T;E)$ be a copy of the graph.
    \State $j:=1$
    \For{$i=1,\dots,n$}
    \While{$\deg(s_i)\neq 0$}
    \State Choose an edge $s_it\in\Delta(s_i)$ for which $x_{s_it}$ is as large as possible.
    \State $e_j:=s_it$
    \State $j:=j+1$
    \State $E:=E\setminus\{s_it\}$
    \EndWhile
    \EndFor
    \State \textbf{output} $e_1,\dots,e_m$
  \end{algorithmic}
\end{algorithm}
    To prove that this order is $(2-\frac{1}{2d-1})$-flat with respect to $x$, let $\xi_i$ and $\bar\xi_i$ $(i=1,\dots,n)$ be as in Definition~\ref{def:flatOrder}. First observe that $\bar\xi_i\leq 1-x_i$ holds for each $i=1,\dots,n$, because the algorithm places each edge $\bigcup_{j=1}^{i-1}\Delta(s^j)$ before $e_i$. Hence, to obtain (\ref{eq:flat}), it suffices to prove that $\xi_i\leq 1-\frac{1}{2d-1}$. For any node $s\in S$, if there exists an edge $st\in\Delta(s)$ for which $x_{st}\geq\frac{1}{2d-1}$, then $\xi_i\leq 1-\frac{1}{2d-1}$ follows for each $e_i\in\Delta(s)$, since $x_{e}\geq\frac{1}{2d-1}$ holds for the first edge $e\in\Delta(s)$ selected by Algorithm~\ref{alg:orderingProc}. Otherwise, there exists no edge $st\in\Delta(s)$ for which $x_{st}\geq\frac{1}{2d-1}$. Therefore $\xi_i\leq x(\Delta(s))<|E_s|\frac{1}{2d-1}\leq 1$ follows for $e_i\in\Delta(s)$, which completes the proof if $|E_i|<2d-1$. Hence one can assume that $|E_i|=2d-1$. Next we argue that $x'=x+\epsilon\chi_{st'}$ is feasible for some $st'\in E_s$ and sufficiently small $\epsilon>0$. By contradiction, if there existed no such edge $st'$, then it is one of the constraints~(\ref{eq:autokDM}) that prevents us from increasing $x_{st'}$ for each $st'\in\Delta(s)$. However, these tight constraints imply that $x(\bigcup\{\Delta(s') : s'\in L_d(s)\cup R_d(s)\})=2d-1$, but this can not be the case, because $x(\Delta(s))<1$. Hence $x'$ is a feasible solution for some $st'\in\Delta(s)$ and sufficiently small $\epsilon>0$ --- contradicting the optimality of $x$.
    
    Therefore $\xi_i\leq 1-\frac{1}{2d-1}$ follows for $i=1,\dots,n$, which means that the order of the edges is $(2-\frac{1}{2d-1})$-flat.
    \qed
  \end{proof}

\begin{thm}\label{thm:gap}
  The integrality gap of \ref{dmLp} is at most $2-\frac{1}{2d-1}$.
\end{thm}
  \begin{proof}
    Let $\theta=2-\frac{1}{2d-1}$. By Lemma~\ref{lem:orderingNEW}, there exists a solution $x\in\Q^E$ to \ref{dmLp} and an order of the edges $e_1=s^1t^1,\dots,e_m=s^mt^m$ that is $\theta$-flat with respect to $x$. First, it will be shown that there exist $d$-distance matchings $M_1,\dots,M_q$ and coefficients $\lambda_1,\dots,\lambda_q\in\Q_+$ such that $\sum_{i=1}^q\lambda_i\chi_{M_i}=x$ and $\lambda:=\sum_{i=1}^q\lambda_i\leq\theta$.
    
    Let $K\in\N$ be the lowest common denominator of $\{x_e : e\in E\}$, and let ${q=\floor{K\theta}}$. The main observation is that each edge $e\in E$ can be assigned a set of colors $C_e\subseteq\{1,\dots,q\}$ such that each color class corresponds to a feasible $d$-distance matching and $|C_e|=Kx_e$. To prove this, the edges are greedily colored one by one in order $e_m,\dots,e_1$.
    By induction, assume that edges $e_m,\dots,e_{i+1}$ already have their color sets. It suffices to assign a color set $C_{e_i}$ to edge $e_i$ which is of size $Kx_{e_i}$ and distinct from both $A:=\bigcup\{C_{e_j} : j>i, e_j\in\Delta(s^i) \}$ and $B:=\bigcup\{C_{e_j} : j>i, e_j\in\Delta(t^i), s^j\in R_d(s^i)\cup L_d(s^i)\}$. Without loss of generality, assume that $x_{e_i}>0$ (otherwise $C_{e_i}=\emptyset$).
    By (\ref{eq:flat}), one gets $|A\cup B|\leq |A|+|B|=K(\xi_i+\bar\xi_i)\leq \lfloor K(\theta-x_{e_i})\rfloor=\lfloor K\theta\rfloor-Kx_{e_i}=q-Kx_{e_i}$,
    thus $|A\cup B|+ Kx_{e_i}\leq q$. That is, the number of free colors is at least $Kx_{e_i}$, so let $C_{e_i}$ be any $Kx_{e_i}$ colors in $\{1,\dots,q\}\setminus(A\cup B)$.
    
    Let the desired $d$-distance matching $M_i$ consist of the edges with color $i$ for $i=1,\dots,q$. Set $\lambda_i=\frac{1}{K}$ for all $i=1,\dots,q$, and observe that both ${\sum_{i=1}^q\lambda_i\chi_{M_i}=x}$ and ${\sum_{i=1}^q\lambda_i=\sum_{i=1}^{q}\frac{1}{K}=\frac{q}{K}\leq\theta}$ hold.
    
    Now, we are ready to argue that there exists a $\lambda$-approximate solution among $M_1,\dots,M_q$. By contradiction, suppose that $\lambda w(M_i)<w(M^*)$ for each $i=1,\dots,q$, where $M^*$ is an optimal distance matching. Observe that
    $w\sum\nolimits_{i=1}^q\lambda_i\chi_{M_i}=\sum\nolimits_{i=1}^q\lambda_iw(M_i)< \tfrac{1}{\lambda}w(M^*)\sum\nolimits_{i=1}^q\lambda_i=w(M^*)$,
    that is, the LP optimum is strictly smaller than the IP optimum, which is a contradiction. Therefore, the largest weight $d$-distance matchings among $M_1,\dots,M_q$ are indeed $\lambda$-approximate. Since $\lambda=\sum_{i=1}^q\lambda_i\leq 2-\frac{1}{2d-1}$, the proof is complete.
    \qed
  \end{proof}

Note that the above approach is algorithmic, but it does not necessarily run in polynomial time --- since $q$ may be exponential in the size of the graph. The next section presents a polynomial-time method and re-proves that the integrality gap is at most $2-\frac{1}{2d-1}$.

  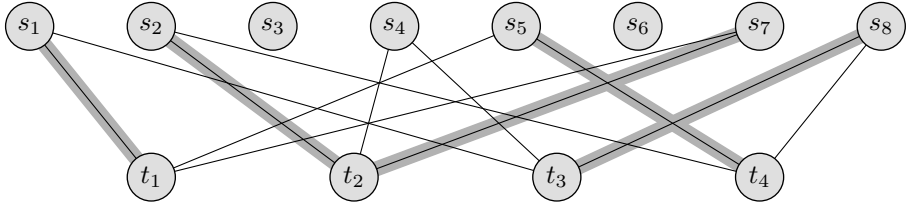
\begin{figure}
  \centering
  \begin{tikzpicture}
    \SetVertexMath
    \grEmptyPath[form=1,x=0,y=2,RA=1.6,rotation=0,prefix=s]{8}
    \Vertex[x=1.6,        y=0,L=t_1]{t0}
    \Vertex[x=1.6+5*1.6/3,  y=0,L=t_2]{t3}
    \Vertex[x=1.6+2*5*1.6/3,y=0,L=t_3]{t1}
    \Vertex[x=1.6+3*5*1.6/3,y=0,L=t_4]{t2}

    \draw[] (t0)--(s0);
    \draw[] (t0)--(s4);
    \draw[] (t0)--(s6);

    \draw[] (t1)--(s0);
    \draw[] (t1)--(s3);
    \draw[] (t1)--(s7);

    \draw[] (t2)--(s1);
    \draw[] (t2)--(s4);
    \draw[] (t2)--(s7);

    \draw[] (t3)--(s1);
    \draw[] (t3)--(s3);
    \draw[] (t3)--(s6);

    \begin{pgfonlayer}{background}
      \highlight{2mm}{black!30}{(t0.center) -- (s0.center)}
      \highlight{2mm}{black!30}{(t1.center) -- (s7.center)}
      \highlight{2mm}{black!30}{(t2.center) -- (s4.center)}
      \highlight{2mm}{black!30}{(t3.center) -- (s1.center)}
      \highlight{2mm}{black!30}{(t3.center) -- (s6.center)}
    \end{pgfonlayer}    

  \end{tikzpicture} 
  \caption{For $w\equiv 1$ and $d=5$, $x\equiv 1/2$ is an optimal solution to \ref{dmLp}, and the high\-lighted edges form an optimal 5-distance matching, hence the integrality gap is $\sfrac{6}{5}$.\looseness=-1}
  \label{fig:ipGapExample}
\end{figure}

\begin{remark}
  Figure~\ref{fig:ipGapExample} provides an example with (the largest known) integrality gap 6/5. Using this instance, one might easily derive an example (by adding two new nodes $t_5$ and $t_6$ to $T$, and two new edges $s_3t_5, s_6t_6$) for which no perfect $5$-distance matching exists, but there is a fractional perfect $5$-distance matching --- meaning that the integrality gap of~LP2 is unbounded as it was expected due to the complexity of the problem.
\end{remark}

In what follows, the integrality of \ref{dmLp} and LP2 
are shown in special cases.

\begin{thm}\label{thm:LP12int}
  If $d=1$ or $d=2$, then both \ref{dmLp} and LP2 
  are integral.
\end{thm}
  \begin{proof}
    For $d=1$, the matrix of \ref{dmLp} and LP2 is the incidence matrix of a bipartite graph, which is a well-known network matrix~\cite[Page 149]{AF11}.
    For $d=2$, one can easily construct a directed graph and a spanning tree (in this case a directed caterpillar) for which the corresponding network matrix is the matrix of \ref{dmLp} and LP2. As the right-hand side of both programs are integral and their matrices are network matrices (and hence totally unimodular) for $d=1,2$, the proof is complete. Note that the statement for $d=1$ also follows from Theorem~\ref{thm:gap}.
    \qed
  \end{proof}

 Note that the matrix of \ref{dmLp} and LP2 
 is not totally unimodular for $d\geq 3$ if the input graph is the complete bipartite graph --- the technical proof is omitted here. Therefore, the proof of Theorem~\ref{thm:LP12int} can not work for $d\geq 3$ and one can not expect that \ref{dmLp} and LP2 remain integral. Having said that, LP2 still describes the convex hull of the integral solutions for $d=|T|$, but not because of total unimodularity:

\begin{thm}
  If $d=|T|$, then LP2 
  is integral.
\end{thm}
  \begin{proof}
    Let $A$ denote the matrix of LP2, and let $\tilde x$ be an optimal integral solution. If $\tilde x$ is not an optimal LP solution, then there is no complementary dual solution $y$, therefore --- by Farkas' lemma --- there exists $z\in\R^{E}$ for which
    \begin{subequations}
      \begin{align}
        \label{javIrBegin}
        wz>0\\
        Az=0\\
        \tilde x_{e}=1\Longrightarrow z_{e}\leq 0 &&\forall e\in E\\
        \tilde x_{e}=0\Longrightarrow z_{e}\geq 0 &&\forall e\in E.
        \label{javIrEnd} 
      \end{align}
    \end{subequations}
    Let $z^j=(z_{s_1t_j},z_{s_2t_j},\dots,z_{s_{|S|}t_j})$. Observe that $z^j_i=z^j_k$ for all $j=1,\dots,d$ whenever $i \equiv k \mod d$, which allows the  simplification of (\ref{javIrBegin})-(\ref{javIrEnd}). For all $i=1,\dots,d$ and $j=1,\dots,|T|$, let $\hat z^j_i$ be a new variable representing all variables $\{z^j_{i'}: i \equiv i' \mod d \}$, and consider the following formulation.
    \begin{subequations}
      \begin{align}
        \label{javIrBeginB}
        \max\hat w\hat z\nonumber\\
        \sum\limits_{i=1}^{d} \hat z^j_i=0 &&\forall j=1,\dots,|T|\\
        \sum\limits_{j=1}^{|T|}\hat z^j_{i}=0 && \forall i=1,\dots,d\\
        \hat z^j_i\leq 0 &&\forall s_it_j\in E : i\in\{1,\dots,d\}\text{ and }\tilde x_{s_it_j}=1\\
        \hat z^j_i\geq 0 &&\forall s_it_j\in E : i\in\{1,\dots,d\}\text{ and }\tilde x_{s_it_j}=0\\
        -\1\leq z\leq \1
        \label{javIrEndB}
      \end{align}
    \end{subequations}
    where $\hat w^j_i=\sum\{w_{i'}^j : i'\in\{1,\dots,|S|\} \text{ and } i'\equiv i\mod d\}$. Note that system (\ref{javIrBegin})-(\ref{javIrEnd}) has a feasible solution if and only if (\ref{javIrBeginB})-(\ref{javIrEndB}) has one with positive objective value. As the optimal value of (\ref{javIrBeginB})-(\ref{javIrEndB}) is finite and its matrix is totally unimodular (the incidence matrix of a bipartite graph and identity matrices under it), there is an integral solution $\hat z^*$ to (\ref{javIrBeginB})-(\ref{javIrEndB}) with a positive objective value. This particular solution corresponds to an integral solution $z^*$ to (\ref{javIrBegin})-(\ref{javIrEnd}) with the same positive weight. But this means that $\tilde x + z^*$ is an integral solution of LP2, for which $w\tilde x < w(\tilde x + z^*)$ holds, contradicting the fact that $\tilde x$ was an optimal integral solution.
    \qed
  \end{proof}
Note that the analogous statement for \ref{dmLp} does not hold.

\subsubsection{$( 2- \frac{1}{2d-1})$-approximation algorithm for the weighted $d$-distance matching}\label{sec:lpApxAlgo}
This section presents an ``almost greedy'' LP-based $(2-\frac{1}{2d-1})$-approxi\-ma\-tion algorithm and re-proves that the integrality gap is at most $\theta:=2-\frac{1}{2d-1}$.

\begin{algorithm}
  \caption{\hspace{0.5cm}$\theta$-approximation algorithm for the weighted distance matching problem}\label{alg:weighted2Approx}
  \begin{algorithmic}
    \State Let $e_1,\dots,e_m$ be a $\theta$-flat order with respect to a solution $x$ of \ref{dmLp} (see Lemma~\ref{lem:orderingNEW}).
    \State \textbf{procedure }\textsc{WdmLpApx}$(E,w)$
    \State $E:=E\setminus\{e\in E : w_e\leq0\}$
    \If{$E=\emptyset$}
    \State \textbf{return} $\emptyset$
    \EndIf
    \State Let $st$ be the first edge according to the above order that appears in $E$.
    \State $M':= \textsc{WdmLpApx}(E\setminus\{st\},w')$, where  $w':=w-w_{st}\chi_{\Delta(s)\cup\{s't\in\Delta(t) : s'\in R_d(s)\}}$
    \If{$M'\cup\{st\}$ is a feasible $d$-distance matching}
    \State \textbf{return } $M'\cup\{st\}$
    \Else
    \State \textbf{return } $M'$
    \EndIf
  \end{algorithmic}
\end{algorithm}

\begin{thm}\label{thm:algGap}
  Algorithm~\ref{alg:weighted2Approx} is a $\theta$-approximation algorithm for the weighted $d$-distance matching problem if a $\theta$-flat order of the edges is given in the first step of the algorithm.
\end{thm}
  \begin{proof}
    The proof is by induction on the number of edges. Let $M$ denote the distance matching found by \textsc{WdmLpApx}(E,w), and let $x$ be as defined in Algorithm~\ref{alg:weighted2Approx}.
    In the base case, if $E=\emptyset$, then $\theta w(M)\geq wx$ holds. Let $st\in E$ be the first edge with respect to the order of the edges used by Algorithm~\ref{alg:weighted2Approx}. By induction,
      $\theta w'(M')\geq w'x$
    holds for $M'= \textsc{WdmLpApx}(E\setminus\{st\},w')$, where $w'=w-w_{st}\,\chi_{\Delta(s)\cup\{s't\in\Delta(t) : s'\in R_d(s)\}}$.
    The key observation is that
    \begin{equation}\label{eq:lpApx2}
      \theta (w-w')(M)\geq \theta w_{st} \geq (w-w')x
    \end{equation}
    follows from the definition of $w'$ and the order of the edges. Hence, one gets
    \begin{equation}
      \theta w(M)=\theta (w-w')(M)+\theta w'(M)\geq (w-w')x+w'x=wx,
    \end{equation}
    where $w'(M)=w'(M')$ because $w'_{st}=0$.
    Therefore, $M$ is indeed a $\theta$-approximate solution, which completes the proof.
    \qed
  \end{proof}
  
Theorem~\ref{thm:algGap} also implies that the integrality gap of \textit{LP1} is at most $\theta$. Note that if we have a $\theta'$-flat order of the edges in the first step of Algorithm~\ref{alg:weighted2Approx}, then it outputs a $\theta'$-approximate solution. We believe that there always exists a $\theta'$-flat order of the edges for some $\theta'<\theta$, i.e. it is possible to improve Lemma~\ref{lem:orderingNEW}, which would automatically improve both the integrality gap and the approximation guarantee of the algorithm to $\theta'$.

\subsection{A combinatorial $(2-\frac{1}{d}$)-approximation algorithm}\label{sec:combApx}
This section pre\-sents a $(2-\frac{1}{d}$)-approximation algorithm for the weighted distance matching problem. Let $k\in\{d-1,\dots,3d-3\}$ be such that $2d-1$ divides $|S|+k$, and add $k$ new dummy nodes $s_{n+1},\dots,s_{n+k}$ to the end of $S$ in this order. Let us consider the extended node set in cyclic order. Observe that the new cyclic problem is equivalent to the original one. Let $H_j$ denote the subgraph of $G$ induced by $R_d(s_{j})\cup T$, where $R_d(s_{j})$ is the set consisting of node $s_{j}$ and the next \mbox{$d-1$ nodes} on its right in the new cyclic problem. For each such subgraph $H_j$, let $F_j$ denote a maximum-weight matching of it with respect to $w$. Let
\begin{equation*}
    G_i=(S_i,T;E_i)=\bigcup\limits_{j=0}^{\frac{n+k}{2d-1}-1} H_{i+j(2d-1)}    
\end{equation*}
and
\begin{equation*}
    M_i=\bigcup\limits_{j=0}^{\frac{n+k}{2d-1}-1} F_{i+j(2d-1)}
\end{equation*}
for $i=1,\dots,2d-1$, where $S_i\subseteq S$. 
Let ${i^*=\argmax\{w(M_i) : i=1,\dots,2d-1\}}$.
For example, consider the graph in Figure~\ref{fig:combApxTight} with $d=3$. The nodes of $G_4$ are highlighted on the figure and the edges of $M_4$ are the wavy ones. Nodes $s_6,\dots,s_{10}$ are the five dummy nodes.

Since $M_{i^*}$ can be computed in strongly polynomial time, we obtain a strongly-polynomial-time $(2-\frac{1}{d})$-approximation algorithm by the following theorem.
\begin{thm}\label{thm:combApx}
    $M_{i^*}$ is a $(2-\frac{1}{d})$-approximate $d$-distance matching.
\end{thm}
\begin{proof}
    Each node of $S$ is covered by at most one edge of $M_i$, as $M_i$ is the union of matchings no two of which cover the same node of $S$. If $s_it,s_jt\in M_i$, then $s_it$ and $s_jt$ belong to two distinct matchings $F_k, F_l\subseteq M_i$ for some $k,l$, hence $|j-i|\geq d$. From this, the feasibility of $M_i$ follows for all $i=1,\dots,2d-1$, and $M_{i^*}$ being one of them, it is feasible as well.

    To show the approximation guarantee, let $M^*$ be an optimal $d$-distance matching. For each node $s\in S$, let $\mu_s\in\R_+$ denote the weight of the edge covering $s$ in $M^*$ and zero if $M^*$ does not cover $s$. Note that $\sum_{s\in S}\mu_s=w(M^*)$ by definition,
    and
    \begin{equation}\label{eq:wcaEq}
        \sum_{s\in S_i}\mu_s\leq w(M_i)
    \end{equation}
    follows because $\sum_{s\in S_i}\mu_s$ is the weight of a $d$-distance matching which covers no nodes outside $G_i$. Observe that
    \begin{equation}\label{eq:wcaComp}
        dw(M^*) = d\sum_{s\in S}\mu_s = \sum_{i=1}^{2d-1}\sum_{s\in S_i}\mu_s\leq \sum_{i=1}^{2d-1}w(M_i)\leq (2d-1)w(M_{i^*})
    \end{equation}
    holds, where the second equation holds because $\mu_s$ occurs $d$ times as a summand in $\sum_{i=1}^{2d-1}\sum_{s\in S_i}\mu_s$ for all $s\in S$, the first inequality follows from~(\ref{eq:wcaEq}), while the last one holds because $M_{i^*}$ is a largest-weight $d$-distance matching among $M_1,\dots,M_{2d-1}$.
    By (\ref{eq:wcaComp}), one gets $w(M^*)\leq (2-\frac{1}{d})w(M_{i^*})$, which completes the proof of the theorem.
    \qed
\end{proof}

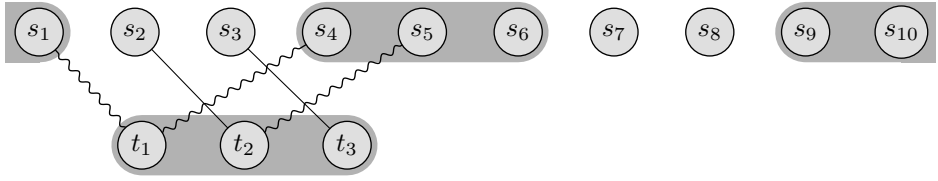
\begin{figure}
\centering
    \begin{tikzpicture}[xscale=.9]
      \SetVertexMath
      \grEmptyPath[form=1,x=0,y=1.5,RA=1.4,rotation=0,prefix=s]{10}
      \grEmptyPath[form=1,x=1.5,y=0,RA=1.5,rotation=0,prefix=t]{3}
      \draw[wavy] (t0)--(s0);
      \draw[] (t1)--(s1);
      \draw[] (t2)--(s2);
      \draw[wavy] (t0)--(s3);
      \draw[wavy] (t1)--(s4);
      
      \begin{pgfonlayer}{background}
      \highlight{8mm}{black!30}{(s0.center) -- (s0.center)}
      \highlight{8mm}{black!30}{(s3.center) -- (s5.center)}
      \highlight{8mm}{black!30}{(s8.center) -- (s9.center)}
      \highlight{8mm}{black!30}{(t0.center) -- (t2.center)}
      \fill [black!30] (9*1.4,1.5-.4) rectangle (9*1.4+.5,1.5+.4);
      \fill [black!30] (-.5,1.5-.4) rectangle (0,1.5+.4);
      \end{pgfonlayer}
    \end{tikzpicture} 
    \caption{Tight example for Theorem~\ref{thm:combApx} in the case $d=3$. The wavy edges form a possible output of the algorithm. (Recall that the nodes of $S$ are in cyclic order.)}
    \label{fig:combApxTight}
\end{figure}
The analysis is tight in the sense that, for every $d\in\Z_+$, there exists a graph $G$ for which the algorithm returns a $d$-distance matching $M$ for which $w(M^*)=(2-\frac{1}{d})w(M)$, where $M^*$ is an optimal $d$-distance matching. Let $S$ and $T$ consist of $2d-1$ and $d$ nodes, respectively. Add edge $s_it_i$ for $i=1,\dots,d$, and edge $s_{i+d}t_i$ for $i=1,\dots,d-1$. Note that the edge set is a feasible $d$-distance matching itself, and the above algorithm returns a matching that covers exactly $d$ nodes of $S$. Hence the approximation ratio of the found solution is $\frac{2d-1}{d}$. Figure~\ref{fig:combApxTight} shows the construction for $d=3$, where $s_6,\dots,s_{10}$ are the dummy nodes.

\section{Unweighted $d$-distance matching}\label{sec:uwdm}
First, two refined greedy approaches are considered for the unweighted case, then the analysis of the approximation ratio of locally optimal solutions follows.

\subsection{Greedy algorithms}\label{sec:uwGreedy}
This section describes two refined greedy algorithms for the unweighted $d$-distance matching problem, and proves that both of them achieve an approximation guarantee of 2.
\begin{algorithm}
  \caption{\hspace{0.5cm}\textsc{$S$-Greedy}}
  \label{alg:greedyS}
  \begin{algorithmic}

    \State Let $s_1,\dots,s_n$ be the nodes of $S$ in the given order 
    \State $M:=\emptyset$
    \For{$i=1,\dots,n$}
    \If{$M\cup\{s_it\}$ is feasible for some $s_it\in\Delta(s_i)$}
    \State $j:=\argmin\{j : s_it_j\in\Delta(s_i)$ and $M\cup\{s_it_j\}$ is feasible $\}$
    \State $M:=M\cup\{s_it_j\}$
    \EndIf
    \EndFor
    \State \textbf{output} $M$
  \end{algorithmic}
\end{algorithm}

\begin{thm}\label{thm:SGreedy}
  \textsc{$S$-Greedy} is a 2-approximation algorithm for the unweighted $d$-distance matching problem.
\end{thm}
  \begin{proof}
    Assume that \textsc{$S$-Greedy} returns edges $f_1,\dots,f_p$, and it selects them in this order.
    Let $M_i$ and $\theta_i$ be as above in the proof of Theorem~\ref{thm:wGreedy}, i.e. let $M_i=\argmax\{w(M) : f_1,\dots,f_i\in M \text{ and $M$ is a $d$-distance matching}\}$ and let $\theta_i$ denote the weight of $M_i$ for $i=0,\dots,p$.
    Observe that, as opposed to the proof of Theorem~\ref{thm:wGreedy}, there exist two edges $e,e'\in M_i\setminus\{f_1,\dots,f_i\}$ such that $(M_i\setminus\{e,e'\})\cup\{f_{i+1}\}$ is a feasible $d$-distance matching containing edges $f_1,\dots,f_{i+1}$. Indeed, if there were three edges to leave out, then one of them would be incident to $\{s_{i-d+1},\dots,s_{i-1}\}$, but then Algorithm~\ref{alg:greedyS} would have picked this edge instead of $f_{i+1}$. By the greedy selection rule, one gets
    \begin{equation}\label{eq:greedyProgressUnw}
      \theta_{i+1}\geq\theta_i+1-1-1=\theta_i-1
    \end{equation}
    holds for all $i=0,\dots,p-1$. A straightforward inductive argument shows that (\ref{eq:greedyProgressUnw}) implies $\theta_p\geq\theta_0-p$, therefore $2\theta_p\geq\theta_0$ follows, which completes the proof.

    The analysis is tight in the sense that \textsc{$S$-Greedy} might return only one edge, while the largest $2$-distance matching consists of 2 edges, see Figure~\ref{fig:unwGreedyTight}.
    \qed
  \end{proof}

\begin{algorithm}
  \caption{\hspace{0.5cm}\textsc{$T$-Greedy}}
  \label{alg:greedyT}
  \begin{algorithmic}

    \State Let $s_1,\dots,s_n$ be the nodes of $S$ in the given order 
    \State $M:=\emptyset$
    \State $C:=\emptyset$
    \For{$j=1,\dots,k$}
    \State $M_j:=\emptyset$
    \State $C_j:=\emptyset$
    \State $i:=1$
    \While{$i\leq n$}
    \If{$s_it_j\in E$ and $s_i\notin C$}
    \State $M_j:=M_j\cup\{s_it_j\}$
    \State $C_j:=C_j\cup\{s_i\}$
    \State $i:=i+d$
    \Else
    \State $i:=i+1$
    \EndIf
    \EndWhile
    \State $M:=M\cup M_j$
    \State $C:=C\cup C_j$
    \EndFor
    \State \textbf{output} $M$
  \end{algorithmic}
\end{algorithm}

\begin{thm}\label{thm:TGreedy}
  \textsc{$T$-Greedy} is a 2-approximation algorithm for the unweighted $d$-distance matching problem.
\end{thm}
  \begin{proof}
    Let $M_S$ and $M_T$ denote the edge sets \textsc{$T$-Greedy} (Algorithm~\ref{alg:greedyT}) and \textsc{$S$-Greedy} (Algorithm~\ref{alg:greedyS}) outputs, respectively. It suffices to prove that $M_S=M_T$. By contradiction, suppose that $M_S\neq M_T$. Let $s_i$ be the first node in $S$ for which $\Delta(s_i)\cap M_S\neq \Delta(s_i)\cap M_T$, and choose the edge $s_it_j\in\Delta(s_i)\cap( M_S\Delta M_T)$ such that $j$ is the smallest possible.

    Case 1: $s_it_j\in M_S\setminus M_T$. First, observe that $M_T$ covers node $s_i$, otherwise it would have included $s_is_t$. Therefore, \textsc{$T$-Greedy} assigns node $s_i$ to $t_{j'}$, where $j'\neq j$. If $j'<j$, then \textsc{$S$-Greedy} would have chosen edge $s_it_{j'}$ instead of $s_it_j$. If $j'>j$, then \textsc{$T$-Greedy} would have included $s_it_j$ instead of $s_it_{j'}$ to $M_T$.

    Case 2: $s_it_j\in M_T\setminus M_S$. Observe that $M_S$ covers node $s_i$, otherwise \textsc{$S$-Greedy} could have included edge $s_is_t$. Therefore, \textsc{$S$-Greedy} assigns node $s_i$ to $t_{j'}$, where $j'\neq j$. Similarly to the argument in Case~1, it is easy to see that neither $j'<j$ nor $j'>j$ is possible.

    Figure~\ref{fig:unwGreedyTight} shows that the approximation ratio is tight.
    \qed
  \end{proof}

\subsection{Local search}\label{sec:localSearch}
This section investigates the approximation ratio of the so-called locally optimal solutions. First, consider the following notion, which plays a central role throughout the section.
\begin{defn}\label{defn:hitset}
  Given an edge $e^*\in E$, let $\hit(e^*,M)\subseteq M$ denote the inclusion-wise minimal subset of $M$ for which $M\setminus\hit(e^*,M)\cup\{e^*\}$ is a feasible $d$-distance matching.
\end{defn}

We say that an edge $e^*$ \emph{hits} the edges of $\hit(e^*,M)$, or that $\hit(e^*,M)$ is the \textit{hit set} of edge $e^*$ with respect to $M$. Similar notation and terminology are used for a subset of the edges as follows.

\begin{defn}
  Given an edge set $X\subseteq E$, let $\hit(X,M)\subseteq M$ denote the set of edges hit by at least one edge in $X$, i.e. let $\hit(X,M) = \bigcup_{e^*\in X}\hit(e^*,M)$.
\end{defn}

\begin{defn}
A $d$-distance matching $M$ is \textbf{$l$-locally optimal} if there exists no $d$-distance matching $X\subseteq E\setminus M$ such that $l\geq |X| > |\mathcal{H}(X,M)|$. Similarly, $M$ is \textbf{$l$-locally optimal with respect to $M^*$} if there exists no $X\subseteq M^*\setminus M$ such that $l\geq |X|>|\mathcal{H}(X,M)|$, where $M^*$ is an $d$-distance matching.
\end{defn}

For unit weights, the possible outputs of \textsc{Greedy} (Algorithm~\ref{alg:greedy}) are exactly the $1$-locally optimal solutions.

\begin{cl}
    A $d$-distance matching $M$ is $1$-locally optimal if and only if there exists a permutation of $E$ such that \textsc{Greedy} outputs $M$ for $w\equiv 1$.
\end{cl}
\begin{proof}
    If $M$ is the output of \textsc{Greedy}, then there exists no edge~$e$ outside $M$ which can be added to $M$ (otherwise \textsc{Greedy} could have added $e$ when it tried to), hence $M$ is $1$-locally optimal by definition. On the other hand, if $M$ is $1$-locally optimal, then permute $E$ such that the edges of $M$ come first. As $w\equiv 1$, one can choose this particular permutation in the first line of Algorithm~\ref{alg:greedy}. To complete the proof, observe that its output is $M$ itself, since it includes all edges of $M$ as $M$ is feasible, and may not include any other edges, because $M$ is $1$-locally optimal.
    \qed
\end{proof}

In what follows, an upper bound $\varrho_l$ is shown on the approximation ratio of $l$-locally optimal solutions for each $l\geq 1$, where $\varrho_l$ is defined by the following recursion.
\begin{equation}\label{cases:rhoReq}
  \varrho_l=\begin{cases}
    3, & \text{if } l=1\\
    2, & \text{if } l=2\\
    \dfrac{4\varrho_{l-2}-3}{2\varrho_{l-2}-1}, & \text{if } l\geq 3.
  \end{cases}
\end{equation}


For $l=1,2,3,4$, the statement can be proved by a simple argument, given below. However, this approach does not seem to work in the general case. The proof of the general case, which is much more involved and quite esoteric, is given after the following theorem.

\begin{thm}\label{thm:locOptSpecial}
  If $M, M^*$ are $d$-distance matchings such that $M$ is $l$-locally optimal with respect to $M^*$, then the approximation ratio $\sfrac{|M^*|}{|M|}$ is at most $\varrho_{l}$, where $l=1,\dots,4$ and $\varrho_{l}$ is as defined above.
\end{thm}
  \begin{proof}
    Let $M^*_i=\{e^*\in M^* : |\hit(e^*,M)|=i\}$ for $i=0,\dots,3$. Note that $M^*_0,M^*_1,M^*_2,M^*_3$ is a partition of $M^*$, and $M^*_0=\emptyset$ since each edge of $M^*$ hits at least one edge of $M$ if $l\geq 1$.
    Since each edge $e\in M$ can be hit by at most three edges of $M^*$, one gets
    \begin{equation}\label{eq:locOptBaseSpecial}
      3|M|\geq\sum\limits_{e^*\in M^*}|\hit_+(e^*,M)|=|M^*_1|+2|M^*_2|+3|M^*_3|.
    \end{equation}
    \emph{Case $l=1$.}

    It easily follows from (\ref{eq:locOptBaseSpecial}) that
    \begin{equation}
      |M^*|=|M^*_1|+|M^*_2|+|M^*_3|\leq |M^*_1|+2|M^*_2|+3|M^*_3|\leq 3|M|.
    \end{equation}
    \emph{Case $l=2$.}
    Similarly,
    \begin{multline}
      2|M^*|=2(|M^*_1|+|M^*_2|+|M^*_3|)\leq|M^*_1|+|M^*_1|+2|M^*_2|+3|M^*_3|\\\leq |M^*_1|+3|M|\leq 4|M|,
    \end{multline}
    where the second inequality follows from (\ref{eq:locOptBaseSpecial}) and the third one holds because $M$ is $2$-locally optimal with respect to $M^*$.
    \\
    \emph{Case $l=3$.}
    For $l=3$, one has to show that $\sfrac{|M^*|}{|M|}\leq\sfrac{9}{5}$. In the following computation, inequality (\ref{eq:locOptBaseSpecial}) is forced with an appropriate coefficient so that the rest admits the application of case $l=1$ to a derived problem instance.
    \begin{multline}\label{eq:locOptNotNightmareL3}
      5|M^*|=5(|M^*_1|+|M^*_2|+|M^*_3|)=2(|M^*_1|+2|M^*_2|+3|M^*_3|)\\+3|M^*_1|+|M^*_2|-|M^*_3|
      \leq 6|M|+3|M^*_1|+|M^*_2|-|M^*_3| \leq 9|M|,
    \end{multline}
    where the first inequality holds by (\ref{eq:locOptBaseSpecial}), while the last one by the following claim.
    \begin{cl}
      If $M$ is 3-locally optimal with respect to $M^*$, then
      \begin{equation}\label{eq:locOptSpecIneqL3}
          |M^*_2|-|M^*_3|\leq 3(|M|-|M_1^*|).
      \end{equation}
      \begin{proof}
        It suffices to show that there exist $d$-distance matchings $\tilde M$, $\tilde M^*$ such that 
        \begin{enumerate}[1)]
            \item $|\tilde M|=|M|-|M^*_1|$,
            \item $|\tilde M^*|=|M^*_{2}|$,
            \item $\tilde M$ is $1$-locally optimal with respect to $\tilde M^*$.
         \end{enumerate}
        Then, condition 3) implies that $|\tilde M^*|\leq 3|\tilde M|$ holds, from which the inequality to be proved follows by substituting 1) and 2).

        Let $\tilde M = M\setminus\hit(M^*_1,M)$ and $\tilde M^*=M^*_2$. Clearly, both 1) and 2) hold. By contradiction, suppose that 3) does not hold, that is, there exists $e^*_1\in\tilde M^*$ such that $\tilde M\cup \{e^*_1\}$ is a feasible $d$-distance matching. By definition, $e^*_1\in M^*_2$, therefore $e^*_1$ hits exactly two edges $e_1,e_2$ in $M$. Neither $e_1$, nor $e_2$ are in $\tilde M$, thus $e_1,e_2\in\hit(M^*_1,M)$, that is $e_j$ is hit by an edge $e^*_{j+1}\in M^*_1$ for $j=1,2$. Note that $e^*_1,e^*_2,e^*_3$ are pairwise distinct edges, and $\hit(\{e^*_1,e^*_2,e^*_3\},M)=\{e_1,e_2\}$, contradicting that $M$ is 3-locally optimal.
        \qed
      \end{proof}
      \emph{Case $l=4$.}
      One has to show that $\sfrac{|M^*|}{|M|}\leq\sfrac{5}{3}$. As in the previous case, inequality~(\ref{eq:locOptBaseSpecial}) will be applied with an appropriate multiplier so that the rest admits the application of case $l=2$ to a derived problem instance.
      \begin{multline}\label{eq:locOptNotNightmareL4}
        6|M^*|=6(|M^*_1|+|M^*_2|+|M^*_3|)= 2(|M^*_1|+2|M^*_2|+3|M^*_3|)+4|M^*_1|+2|M^*_2|\\
        \leq 6|M|+4|M^*_1|+2|M^*_2|\leq 10|M|,
      \end{multline}
      where the first inequality holds by (\ref{eq:locOptBaseSpecial}), the last one by the following claim.

      \begin{cl}
         If $M$ is 4-locally optimal with respect to $M^*$, then
       \begin{equation}\label{eq:locOptSpecIneqL4}
          2|M^*_2|\leq 4(|M|-|M_1^*|).
      \end{equation}
        \begin{proof}
          It suffices to show that there exist $d$-distance matchings $\tilde M$, $\tilde M^*$ such that
          \begin{enumerate}[1)]
            \item $|\tilde M|=|M|-|M^*_1|$,
            \item $|\tilde M^*|=|M^*_{2}|$,
            \item $\tilde M$ is $2$-locally optimal with respect to $\tilde M^*$.
          \end{enumerate}
          Then, condition 3) implies that $|\tilde M^*|\leq 2|\tilde M|$ holds, from which one obtains the inequality to be proved by substituting 1) and 2).

          Let $\tilde M = M\setminus\hit(M^*_1,M)$ and $\tilde M^*=M^*_2$. As in the proof of the claim in case $l=3$, one might show that 3) holds, hence the desired inequality follows.
          \qed
        \end{proof}
      \end{cl}
    \end{cl}
    
    This concludes the proof of the theorem.
    \qed
\end{proof}

It is worth noting that the proof for $l=3,4$ refers inductively to the case $l-2$, which is quite unexpected. The same idea does not seem to work for $l=5$. Based on cases $l=1,2,3,4$, one gains the following analogous computation.
\begin{multline}
  13|M^*|=13(|M^*_1|+|M^*_2|+|M^*_3|)=4(|M^*_1|+2|M^*_2|+3|M^*_3|)+9|M^*_1|\\
  +5|M^*_2|+|M^*_3|\leq 12|M|+9|M^*_1|+5|M^*_2|+|M^*_3|\leq 21|M|,
\end{multline}
where the last inequality requires that $5|M^*_2|+|M^*_3|\leq 9(|M|-|M^*_1|)$. However, the latter inequality does not admit a constructive argument similar to the cases $l=3,4$ (see the proof of Theorem~\ref{thm:locOptSpecial}). To overcome this complication, consider the following extended problem setting, which surprisingly does admit a constructive argument.
\begin{defn}
  Let $R$ be a set of (parallel) loops on the nodes of $S$. A subset $M\subseteq E\cup R$ is \textbf{(R,d)-distance matching} if it is the union of a $d$-distance matching and $R$.
\end{defn}

Consider the following extension of Definition~\ref{defn:hitset}.

\begin{defn}\label{def:hitPlus}
  Given an $(R,d)$-distance matching $M$ and an edge $sv\in (S\times T)\cup R$, let $$\hit_+(sv,M)=\begin{cases}
    \hit(sv,M\setminus R)\cup\{\text{$e\in R$ : $e$ is incident to node $s$}\}, & \text{if } sv\in S\times T\\
    sv, & \text{if } sv\in R.
  \end{cases}$$
\end{defn}

In other words, each $st\in E$ hits the edges hit by $\hit(st,M)$ and all the loops incident to node $s$, while each loop hits only itself. A natural way to define the hit set of multiple edges is as follows.

\begin{defn}
  Given an edge set $X\subseteq E$, let $\hit_+(X,M) = \bigcup_{e\in X}\hit_+(e,M)$.
\end{defn}

Using $\mathcal{H_+}$, the definition of $l$-locally optimal $d$-distance matchings can be naturally extended to $(R,d)$-distance matchings.

\begin{defn}
  An $(R,d)$-distance matching $M$ is \textbf{$l$-locally optimal} if there exists no $d$-distance matching $X\subseteq E\setminus M $ such that $l\geq |X| > |\mathcal{H_+}(X,M)|$. Similarly, $M$ is \textbf{$l$-locally optimal with respect to $M^*$} if there exists no $X\subseteq M^*\setminus M$ such that $l\geq |X|>|\mathcal{H_+}(X,M)|$, where $M^*$ is an $(R,d)$-distance matching.
\end{defn}

Note that each of these definitions reduces to its original counterpart if $R=\emptyset$. Therefore, it suffices to show that $\varrho_l$ is an upper bound on the approximation ratio of $(R,l)$-locally optimal solutions.

To elaborate on the intuition behind these technical definitions and to understand how $R$ influences locally optimality, suppose that we are given a feasible $d$-distance matching $M$, which we want to make $l$-locally optimal. To this end, one needs to find a $d$-distance matching $X\subseteq E\setminus M$ of cardinality at most $l$ that hits strictly fewer edges of $M$ then its cardinality. In an $(R,d)$-distance matching, however, the number of edges hit by such a subset $X$ can be larger because of the loops (as $\hit_+(X,M)$ also counts those), meaning that the requirements for $l$-locally optimality are relaxed. Intuitively, the loops incident to a node $s\in S$ can be thought of as the "resistances" of $s$: the more loops $s$ has, the less we want to replace the edge of $M$ incident to $s$ with some other edge of $\Delta(s)$. Note, however, that the loops also contribute to the size of the matching, which will be crucial in the proof of the next theorem.

\begin{thm}\label{thm:mainLocOpt}
  If $M, M^*$ are $(R,d)$-distance matchings such that $M$ is $l$-locally optimal with respect to $M^*$, then the approximation ratio $\sfrac{|M^*|}{|M|}$ is at most $\varrho_l$, where $l\geq 1$ and $\varrho_l$ is as defined above.
\end{thm}
  \begin{proof}
    As in the proof of Theorem~\ref{thm:locOptSpecial}, let $M^*_i=\{e^*\in M^* : |\hit_+(e^*,M)|=i\}$ for $i\in \N$, and let $M^*_{i+}=\bigcup_{k=i}^\infty M^*_k$. Note that $M^*_0,M^*_1,\dots$ is a partition of $M^*$, for which $R\subseteq M^*_1$ by definition, and $M^*_0=\emptyset$ since each edge of $M^*$ hits at least one edge of $M$ if $l\geq 1$.
    Similar to (\ref{eq:locOptBaseSpecial}), observe that each edge $e\in M$ can be hit by at most three edges of $M^*$, therefore
    \begin{equation}\label{eq:locOptBase}
      3|M|\geq\sum\limits_{e^*\in M^*}|\hit_+(e^*,M)|=\sum\limits_{k=1}^\infty k|M_k^*|.
    \end{equation}
    The proof is by induction on $l$. The argument for $l=1,2$ is analogous to that in the proof of Theorem~\ref{thm:locOptSpecial}.
    \\
    \emph{Case 1: $l=1$.}

    It easily follows from (\ref{eq:locOptBase}) that
    \begin{equation}
      |M^*|=\sum\limits_{k=1}^\infty|M^*_k|\leq\sum\limits_{k=1}^\infty k|M^*_k|\leq 3|M|.
    \end{equation}
    \\
    \emph{Case 2: $l=2$.}
    Similarly,
    \begin{equation}
      2|M^*|=2\sum\limits_{k=1}^\infty|M^*_k|\leq|M^*_1|+\sum\limits_{k=1}^\infty k|M^*_k|\leq |M^*_1|+3|M|\leq 4|M|,
    \end{equation}
    where the second inequality follows from (\ref{eq:locOptBase}) and the third one holds because $M$ is $2$-locally optimal with respect to $M^*$.
    \\
    \emph{Case 3: $l\geq 3$.}
    One has to show that $\sfrac{|M^*|}{|M|}\leq\sfrac{4\varrho_{l-2}-3}{2\varrho_{l-2}-1}$.
    First, introduce the notation $\alpha(M,M^*)=\sum_{k=3}^\infty (k-2)|M^*_k|$.
    In the following computation, inequality (\ref{eq:locOptBase}) is forced with an appropriate multiplier so that the rest admits the application of case $l-2$ to a derived problem instance (see Lemma~\ref{lem:locOpt}). Note that the approach is similar to computations~(\ref{eq:locOptNotNightmareL3})~and~(\ref{eq:locOptNotNightmareL4}).
    \begin{multline}\label{eq:locOptNightmare}
      (2\varrho_{l-2}-1)|M^*|=(2\varrho_{l-2}-1)\sum\limits_{k=1}^\infty |M^*_k|=(\varrho_{l-2}-1)\sum\limits_{k=1}^\infty k|M^*_k|\\
      +\sum\limits_{k=1}^\infty ((k-1)-(k-2)\varrho_{l-2})|M^*_k|\\
      \leq 3(\varrho_{l-2}-1)|M|+\sum\limits_{k=1}^\infty ((k-1)-(k-2)\varrho_{l-2})|M^*_k|\\
      =3(\varrho_{l-2}-1)|M|+\varrho_{l-2}|M^*_1|+|M^*_2|+\sum\limits_{k=3}^\infty (k-1)|M^*_k|-\varrho_{l-2}\sum\limits_{k=3}^\infty (k-2)|M^*_k|\\
      =3(\varrho_{l-2}-1)|M|+\varrho_{l-2}|M^*_1|+|M^*_2|+\sum\limits_{k=3}^\infty |M^*_k|+\alpha(M,M^*)-\varrho_{l-2}\alpha(M,M^*)\\
      =3(\varrho_{l-2}-1)|M|+\varrho_{l-2}|M^*_1|+|M^*_2|+|M^*_{3+}|+\alpha(M,M^*)-\varrho_{l-2}\alpha(M,M^*)\\
      \leq 3(\varrho_{l-2}-1)|M|+\varrho_{l-2}|M|=(4\varrho_{l-2}-3)|M|,
    \end{multline}
    where the first inequality holds by~(\ref{eq:locOptBase}) and the last one by the following lemma.
    Note that if $R=\emptyset$, then $\alpha(M,M^*)=|M_3^*|$ and $|M_k|=0$ for $k\geq 4$, hence~(\ref{eq:locOptNightmare})~gives back (\ref{eq:locOptNotNightmareL3}) and (\ref{eq:locOptNotNightmareL4}) for $l=3,4$, respectively. The following lemma completes the proof of~(\ref{eq:locOptNightmare}).
    \begin{lem}\label{lem:locOpt}
      If $l\geq 3$ and $M,M^*,\alpha(M,M^*)$ are as above, then
      \begin{equation}\label{eq:locOptIneq}
        |M^*_{2+}|+\alpha(M,M^*)\leq\varrho_{l-2}(|M|-|M^*_1|+\alpha(M,M^*))
      \end{equation}
     \end{lem}
      \begin{proof}
       It suffices to show that if $M$ is $l$-locally optimal with respect to $M^*$, then there exist $\tilde M$, $\tilde M^*$ and $\tilde R$ such that
        
        \begin{enumerate}[1)]
            \item $\tilde M$ and $\tilde M^*$ are $(\tilde R,d)$-distance matchings,
            \item $|\tilde M|=|M|-|M^*_1|+\alpha(M,M^*)$,
            \item $|\tilde M^*|=|M^*_{2+}|+\alpha(M,M^*)$,
            \item $|\tilde R|=\alpha(M,M^*)$,
            \item $\tilde M$ is $(l-2)$-locally optimal with respect to $\tilde M^*$.
        \end{enumerate}
        Then, condition 5) implies that $|\tilde M^*|\leq\varrho_{l-2}|\tilde M|$ holds by induction, from which one obtains (\ref{eq:locOptIneq}) by substituting 2) and 3). We define $\tilde R,\tilde M$ and $\tilde M^*$ such that
        \begin{gather}
            \tilde R = \bigcup_{s^*t^*\in M^*_{3+}}\{|\hit_+(s^*t^*,M)|-2\text{ parallel loops incident to }s^*\},\nonumber\\
            \tilde M = M\setminus\hit_+(M^*_1,M)\cup\tilde R\text{ and }\nonumber\\
            \tilde M^* = M^*_{2+}\cup\tilde R.\nonumber
        \end{gather}
        It is easy to see that $\tilde M, \tilde M^*$ and $\tilde R$ fulfill 1)-4). In the rest of the proof, we argue that 5) holds, as well. By contradiction, suppose that 5) does not hold, that is, there exists $Z\subseteq\tilde M^*:l-2\geq|Z|>|\hit_+(Z,\tilde M)|$. Assume that the instance of the problem at hand is minimal in the sense that $|M|+|M^*|+|\tilde M|+|\tilde M^*|+|Z|$ is minimal. First, various useful properties of minimal problem instances are derived. Note that $|Z|=|\hit_+(Z,M)|+1$ can be assumed, otherwise $|Z|>|\hit_+(Z,M)|+1$ and therefore one could have removed an arbitrary edge from $Z$.

        Observe that if an edge $e\in\hit_+(Z,\tilde M)$ were hit by a sole edge $e^*\in Z$, then $l-2\geq|Z\setminus\{e^*\}|>|\hit_+(Z\setminus\{e^*\},\tilde M)|$ would hold, i.e. one could have left out $e^*$ from $Z$. Therefore, each edge $e\in\hit_+(Z,\tilde M)$ is hit by at least two edges of $Z$. This also implies that $s^*t^*\in Z$ if and only if $\{e\in R : e\text{ is incident to }s^*\}\subseteq Z$. Using this, $Z=\tilde M^*$ follows, because removing all edges $\tilde M^*\setminus Z$ from $M^*$ and all those loops from $R$ that are incident to the removed edges, one obtains a smaller instance (where $\alpha(M,M^*),\tilde R, \tilde M$ and $\tilde M^*$ need to be adjusted appropriately after the edge-removal), which satisfies 1)-4) but not 5)
        .
        Clearly, $M$ remains $l$-locally optimal with respect to $M^*$ after the edge-removal. So, one can assume that $Z=\tilde M^*$.

        A minimal instance also fulfills that there exist no edges $e\in M$ and $e^*\in M^*_1$ such that $e$ is not hit by any edge of $M^*_{2+}$,
        (that is, $\hit_+(e^*,M)\setminus\hit_+(M^*_{2+},M)=\emptyset$),
        otherwise the removal of $e$ and $e^*$ results in a smaller instance satisfying 1)-4) but not 5). Observe that after the removal, $M$ remains $l$-locally optimal with respect to $M^*$, because there exists no $X\subseteq M^*\setminus \{e^*\}$ such that $e\in\hit_+(X,M)$ (since $e$ is not hit by any edge of $M^*_{2+}$), therefore if the new instance were not $l$-locally optimal, then the original instance would not have been either. So, one can assume that $\hit_+(M_1^*,M)\setminus\hit_+(M^*_{2+},M)=\emptyset$.
        
        Now we are ready to derive that $|\hit_+(M^*,M)|<|M^*|\leq l$ holds --- contradicting that $M$ is $l$-locally optimal. The first inequality is shown by
        \begin{multline}\label{cl:locOptOneHand}
          |\hit_+(M^*,M)|=|\hit_+(M^*_{2+},M)|=|\hit_+(M^*_{2+},M)\cap\hit_+(M^*_1,M)|\\
          +|\hit_+(M^*_{2+},M)\setminus\hit_+(M^*_1,M)|=|\hit_+(M^*_1,M)|+|\hit_+(M^*_{2+},M)\setminus\hit_+(M^*_1,M)|\\
          =|M^*_1|+|\hit_+(M^*_{2+},M)\setminus\hit_+(M^*_1,M)|=|M^*_1|+|\hit_+(M^*_{2+},M\setminus\hit_+(M^*_1,M))|\\
          =|M^*_1|+|\hit_+(M^*_{2+},\tilde M\setminus\tilde R)|=|M^*_1|+|\hit_+(Z,\tilde M)\setminus\tilde R|=|M^*_1|+|Z|-1-|\tilde R|\\=|M^*_1|+|\tilde M^*|-1-|\tilde R|=|M^*_1|+|M^*_{2+}|+|\tilde R|-1-|\tilde R|=|M^*|-1.
        \end{multline}
        
        Next, we show that $|M^*|\leq l$.
        \begin{multline}\label{cl:locOptOtherHand}
          |M^*|=|M^*_{2+}|+|M^*_1|=|M^*_{2+}|+|\hit_+(M^*_1,M)|\\
          =|M^*_{2+}|+|\hit_+(M^*_{2+},M)\cap\hit_+(M^*_1,M)|\\
          =|M^*_{2+}|+|\bigcup\limits_{e^*\in M^*_{2+}}\hit_+(e^*,M)\cap\hit_+(M^*_1,M)|\\
          \leq |M^*_{2+}|+\sum\limits_{e^*\in M^*_{2+}}|\hit_+(e^*,M)\cap\hit_+(M^*_1,M)|\\
          = |M^*_{2+}|+\sum\limits_{e^*\in M^*_{2+}}(|\hit_+(e^*,M)|-|\hit_+(e^*,M)\setminus\hit_+(M^*_1,M)|)\\
          \leq |M^*_{2+}|+\sum\limits_{e^*\in M^*_{2+}}|\hit_+(e^*,M)|-2|\hit_+(M^*_{2+},M)\setminus\hit_+(M^*_1,M)|\\
          =|M^*_{2+}|+\sum\limits_{e^*\in M^*_{2+}}|\hit_+(e^*,M)|-2(|\hit_+(Z,\tilde M)\setminus\tilde R|)\\
          =|M^*_{2+}|+\sum\limits_{e^*\in M^*_{2+}}|\hit_+(e^*,M)|-2(|M^*_{2+}|-1)\\
          =|M^*_{2+}|+2|M^*_{2+}|+|\tilde R|-2(|M^*_{2+}|-1)\\
          =|M^*_{2+}|+|\tilde R|+2=|\tilde M^*|+2=|Z|+2\leq l,
        \end{multline}
        where the second inequality holds by the following computation.
        \begin{multline}
          2|\hit_+(M^*_{2+},M)\setminus\hit_+(M^*_1,M)|=2|\hit_+(M^*_{2+},M\setminus\hit_+(M^*_1,M))|\\
          =2|\hit_+(M^*_{2+},\tilde M\setminus\tilde R)|
          =2|\hit_+(\tilde M^*,\tilde M\setminus\tilde R)|\\
          =2|\hit_+(Z,\tilde M\setminus\tilde R)|\leq\sum\limits_{e^*\in Z} |\hit_+(e^*,\tilde M\setminus\tilde R)|\\
          =\sum\limits_{e^*\in\tilde M^*_{2+}}|\hit_+(e^*,M\setminus\hit_+(M^*_1,M))|
          =\sum\limits_{e^*\in\tilde M^*_{2+}}|\hit_+(e^*,M)\setminus\hit_+(M^*_1,M)|,
        \end{multline}
        where the inequality holds because each edge of $\hit_+(Z,\tilde M)$ is hit at least twice by~$Z$.
        Combining (\ref{cl:locOptOneHand}) and (\ref{cl:locOptOtherHand}), one obtains that $|\hit_+(M^*,M)|<|M^*|\leq l$, which contradicts that $M$ is $l$-locally optimal with respect to $M^*$. Hence --- in contrast to the indirect assumption --- condition 5) holds, and this proves the lemma.
        \qed
      \end{proof}
    Note that if $R=\emptyset$, then Lemma~\ref{lem:locOpt} gives back (\ref{eq:locOptNotNightmareL3}) and (\ref{eq:locOptNotNightmareL4}) for $l=3,4$, respectively.
    By Lemma~\ref{lem:locOpt}, inequality (\ref{eq:locOptNightmare}) follows, meaning that the desired recursion~(\ref{cases:rhoReq}) gives a valid upper bound on the approximation ratio of the $l$-locally optimal solutions.
    \qed
  \end{proof}

\begin{cor}\label{cor:mainLocOpt}
  The approximation ratio of $l$-locally optimal $d$-distance matchings is at most $\varrho_l$, where $\varrho_l$ is as defined above.
\end{cor}
  \begin{proof}
    Let $M^*$ denote an optimal $d$-distance matching. By definition, $M$ is $l$-locally optimal with respect to $M^*$, therefore $M$ is $(\emptyset,l)$-locally optimal with respect to $M^*$. By Theorem~\ref{thm:mainLocOpt}, one gets $\sfrac{|M^*|}{|M|}\leq\varrho_l$, which completes the proof.
    \qed
  \end{proof}

\begin{cor}
  For any constant $\epsilon>0$, there exist a polynomial-time algorithm for the unweighted $d$-distance matching problem that achieves an approximation guarantee of $3/2+\epsilon$.
\end{cor}
  \begin{proof}
    By Corollary~\ref{cor:mainLocOpt}, the approximation ratio of $l$-locally optimal solutions is at most $\varrho_l$. One might easily show that $\lim_{l\to\infty}\varrho_l=3/2$. Hence for any $\epsilon>0$, there exists $l_0\in\N$ such that $\varrho_l\leq 3/2+\epsilon$. To complete the proof, observe that~$l_0$ is independent from the problem size, therefore one can compute an $l_0$-locally optimal solution in polynomial time. Note that the number of improvements is at most the size of the matching and hence it is polynomial as well.
    \qed
  \end{proof}




\begin{figure}
  \centering
  \begin{tikzpicture}[xscale=.9]
    \SetVertexMath
    \grEmptyPath[form=1,x=0,y=1.5,RA=1.5,rotation=0,prefix=s]{4}
    \grEmptyPath[form=1,x=1.5,y=0,RA=1.5,rotation=0,prefix=t]{2}
    \draw[] (t0)--(s0);
    \draw[] (t0)--(s2);
    \draw[] (t1)--(s1);
    \draw[] (t1)--(s3);

    \draw[wavy] (t0) -- (s1);
    \draw[wavy] (t1) -- (s2);

  \end{tikzpicture} 
  \caption{The wavy edges form a 2-locally optimal $2$-distance matching $M$, and $M^*=E\setminus M$ is the optimal $2$-distance matching. The approximation ratio $\sfrac{|M^*|}{|M|}$ is $\varrho_2=2$.}
  \label{fig:2LocalTight}
\end{figure}
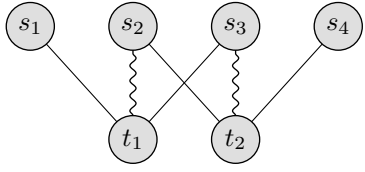

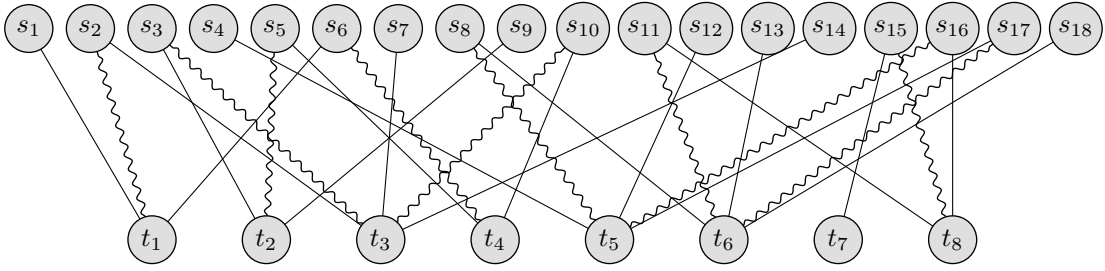
\begin{figure}
  \centering
  \begin{tikzpicture}[xscale=.9]
    \SetVertexMath
    \grEmptyPath[form=1,x=0,y=2.8,RA=.9,rotation=0,prefix=s]{18}

    \pgfmathsetmacro{\distanceInT}{13*.9/7}
    \pgfmathsetmacro{\xOffsetT}{2*.9}
    \Vertex[x=\xOffsetT,        y=0,L=t_1]{t1}
    \Vertex[x=\xOffsetT+\distanceInT,  y=0,L=t_2]{t0}
    \Vertex[x=\xOffsetT+2*\distanceInT,y=0,L=t_3]{t7}
    \Vertex[x=\xOffsetT+3*\distanceInT,y=0,L=t_4]{t5}
    \Vertex[x=\xOffsetT+4*\distanceInT,y=0,L=t_5]{t2}
    \Vertex[x=\xOffsetT+5*\distanceInT,y=0,L=t_6]{t6}
    \Vertex[x=\xOffsetT+6*\distanceInT,y=0,L=t_7]{t4}
    \Vertex[x=\xOffsetT+7*\distanceInT,y=0,L=t_8]{t3}
    
    \draw[]     (t0) -- (s2);
    \draw[wavy] (t0) -- (s4);
    \draw[]     (t0) -- (s8);

    \draw[]     (t1) -- (s0);
    \draw[wavy] (t1) -- (s1);
    \draw[]     (t1) -- (s5);

    \draw[]     (t2) -- (s3);
    \draw[wavy] (t2) -- (s7);
    \draw[]     (t2) -- (s11);
    \draw[wavy] (t2) -- (s15);
    \draw[]     (t2) -- (s16);

    \draw[]     (t3) -- (s10);
    \draw[wavy] (t3) -- (s14);
    \draw[]     (t3) -- (s15);

    \draw[]     (t4) -- (s14);

    \draw[]     (t5) -- (s4);
    \draw[wavy] (t5) -- (s5);
    \draw[]     (t5) -- (s9);

    \draw[]     (t6) -- (s7);
    \draw[wavy] (t6) -- (s10);
    \draw[]     (t6) -- (s12);
    \draw[wavy] (t6) -- (s16);
    \draw[]     (t6) -- (s17);

    \draw[]     (t7) -- (s1);
    \draw[wavy] (t7) -- (s2);
    \draw[]     (t7) -- (s6);
    \draw[wavy] (t7) -- (s9);
    \draw[]     (t7) -- (s13);
    
  \end{tikzpicture} 
  \caption{The wavy edges form a 3-locally optimal $5$-distance matching $M$, and $M^*=E\setminus M$ is the optimal $5$-distance matching. The approximation ratio $\sfrac{|M^*|}{|M|}$ is $\varrho_3=\sfrac{9}{5}$.}
  \label{fig:3LocalTight}
\end{figure}

\begin{remark}
  Figure~\ref{fig:wGreedyTight},~\ref{fig:2LocalTight}~and~\ref{fig:3LocalTight} show that the upper bound on the approximation ratio of $l$-locally optimal solutions given by Theorem~\ref{thm:mainLocOpt} is tight for $l=1,2$ and $3$, respectively. It remains open whether the analysis is tight for $l\geq 4$.
\end{remark}

\begin{remark}
  Similar proof shows that for any constant $\epsilon>0$, the above local-search algorithm is a $(3/2+\epsilon)$-approximation algorithm for the unweighted cyclic \mbox{$d$-distance matching problem}.
\end{remark}

\section{Regular distance matching}\label{sec:regular}
The following theorem is a straightforward generalization of the well-known result that every regular bipartite graph has a perfect matching.
\begin{defn}
  An instance of the $d$-distance matching problem is \emph{$r$-regular} if $\deg(s)=r$ for each $s\in S$ and the number of edges between $t$ and $R_d(s_i)$ is $r$ for each $t\in T$ and $i=1,\dots,n-d+1$.
\end{defn}
\begin{thm}
  If a problem instance is $r$-regular, then there exists a perfect $d$-dis\-tan\-ce matching.
\end{thm}
  \begin{proof}
    There exists a perfect matching between $\{s_1,\dots,s_d\}$ and $T$, because the induced graph is $r$-regular. By induction, assume that the degrees of $\{s_1,\dots,s_{i-1}\}$ in $M$ are one, where $i-1\geq d$. Let $t$ denote the node that $M$ assigns to $s_{i-d+1}$. If $s_it\not\in E$, then the number of edges between $t$ and $L_d(s_i)$ is $r-1$, meaning that the instance at hand is not $r$-regular, hence $s_it\in E$. Therefore, $M\cup\{s_it\}$ is feasible for the first $i$ nodes of $S$, hence the claim follows.
    \qed
  \end{proof}
  
If we leave out a perfect $d$-distance matching from an $r$-regular problem instance, then it becomes an $r-1$ regular instance, hence one gets the following generalization of K\H onig’s edge-coloring theorem~\cite[Page 74]{AF11}.

\begin{cor}
  If a problem instance is $r$-regular, then the edge set of the graph partitions into $r$ perfect $d$-distance matchings.
\end{cor}

\section{Conclusion}
This paper introduced the $d$-distance matching problem. We proved that the problem is NP-complete in general and admits a $3$-approximation. We gave an FPT algorithm parameterized by $d$ and also settled the case when the size of $T$ is constant. The integrality gap of the natural integer programming model is shown to be at most $2-\frac{1}{2d-1}$, and an LP-based approximation algorithm for the weighted case is given with the same guarantee. Using a different approach, a combinatorial ($2-\frac{1}{d}$)-approximation algorithm was also described. Several greedy approaches, including a local search algorithm, were presented. The latter method achieves an approximation ratio of $3/2+\epsilon$ for any constant $\epsilon>0$ in the unweighted case.

The problem itself has several generalizations (e.g. pose degree bounds on the nodes of both $S$ and $T$, cyclic version of the problem, distance-constraints on both node classes, etc.), which are subjects for further research.

\section*{Acknowledgement}
The project has been supported by the \'UNKP-20-3 New National Excellence Program of the Ministry for Innovation and Technology.
The author is grateful to Krist\'of B\'erczi and Alp\'ar J\" uttner for their comments that greatly improved an earlier version of the manuscript.
The author is indebted to the anonymous reviewers of an earlier version of the manuscript for their valuable comments and suggestions.

\bibliographystyle{abbrv}
\bibliography{references}   

\end{document}